\documentclass[a4paper,10pt]{amsart}

\usepackage[margin=2.7cm]{geometry}
\usepackage[foot]{amsaddr}
\usepackage{fancyhdr}
\usepackage{parskip}

\usepackage{amsmath}
\usepackage{amsfonts,amssymb,amsthm}
\usepackage{mathtools}
\mathtoolsset{showonlyrefs=true, showmanualtags=true}
\usepackage{isomath}
\usepackage{mathrsfs}
\usepackage{stmaryrd}         
\usepackage{systeme}
\usepackage{stackrel}

\usepackage{bm}
\usepackage{textgreek}
\usepackage{units}
\usepackage{soul}
\usepackage{cancel}

\usepackage{graphicx}
\usepackage{color}
\usepackage{subcaption}
\usepackage{float}

\usepackage{booktabs}
\usepackage{array}
\usepackage{multicol}
\usepackage{multirow}
\usepackage{tabularx}

\usepackage{enumitem}
\usepackage{tcolorbox}

\usepackage{algorithm,algorithmic}

\usepackage{listings}

\usepackage[titletoc]{appendix}

\usepackage{natbib}

\usepackage[colorlinks,allcolors=black]{hyperref}

\usepackage{cleveref}

\theoremstyle{plain}
\newtheorem{theorem}{Theorem}[section]
\newtheorem{proposition}[theorem]{Proposition}
\newtheorem{lemma}[theorem]{Lemma}

\theoremstyle{definition}
\newtheorem{definition}[theorem]{Definition}
\newtheorem{remark}[theorem]{Remark}

\crefname{assumption}{assumption}{assumptions}

\newcommand{\R}{\mathbb{R}}
\newcommand{\Th}{\mathcal{T}_h}
\newcommand{\Fh}{\mathcal{F}_h}
\newcommand{\Omu}{\Omega_\mu}
\newcommand{\Gmu}{\Gamma_\mu}
\newcommand{\Vh}{V_h}
\newcommand{\norm}[2]{\left\|#1\right\|_{#2}}
\newcommand{\bnorm}[1]{\left\|#1\right\|}
\newcommand{\jump}[1]{\llbracket #1 \rrbracket}
\DeclareMathOperator{\diag}{diag}

\definecolor{codebg}{rgb}{0.97,0.97,0.97}
\definecolor{codegreen}{rgb}{0,0.50,0}
\definecolor{codeblue}{rgb}{0.1,0.1,0.7}

\begin{document}
\title[{{{A} {{P}}osteriori errors for {{P}}arameter {{D}}ependent {G}eometries}}]{{{A} {P}osteriori {E}rror {A}nalysis,
		{P}od-{D}eim 
		{R}educed {O}rder geometrically parametrized {M}odels 
		and 
		{{U}}nfitted {{F}}ems}}
\author{Efthymios N. Karatzas\textsuperscript{$\dagger$,$\ddagger$
}}
\address{\textsuperscript{$\dagger$}School of Mathematics, Aristotle University of Thessaloniki, Thessaloniki 54124, Greece
.}
\address{\textsuperscript{$\ddagger$}SISSA (Αffil.), International School for Advanced Studies, Mathematics Area/mathLab, Trieste, Italy.}
\email{ekaratza@math.auth.gr
}
\subjclass[2000]{Primary}%
\date{21 April 2026}

\thanks{MSCcodes: 65N15, 65N30, 65N75, 65N12}%
\begin{abstract}
We develop and analyze a posteriori error estimators for a
proper orthogonal decomposition–discrete empirical interpolation method
(Pod-Deim) reduced order model applied to a parametric Poisson equation posed on a parameter-dependent domain defined by a level-set function. The full-order discretisations employ a cut finite element method (Cutfem) with Nitsche boundary conditions and ghost-penalty stabilization. Three complementary estimators are proposed: (i) Deim approximation quality indicators for the stiffness matrix and force vector, which are constant in the number of Pod modes, (ii) dual-norm residual estimators in both plain and Jacobi-preconditioned form, and (iii) a Pod tail-energy indicator. A rigorous theoretical framework is established, comprising a uniform coercivity result for the Cutfem bilinear form, an active-dof residual bound that accounts for ghost-penalty degrees of freedom, a combined a posteriori bound, and sharp effectivity analysis for the residual estimators. The key theoretical finding is that the large observed effectivity indices are explained by ghost-penalty degree-of-freedom inflation, and that restricting the residual to active degrees of freedom is predicted to reduce effectivity. Numerical experiments on a parametric ellipse domain with semi-axes confirm the theoretical predictions, achieve significant online speedup, and demonstrate algebraic convergence of the true error alongside exponential decay of the residual estimators.
\end{abstract}
\maketitle
\section{Introduction}
\label{sec:intro}
Motivation and reduced order modeling context started from
the numerical simulation of parametric partial differential equations
(pde) is central to engineering design, uncertainty quantification,
and real-time control. In settings where the pde must be solved repeatedly across a large number of parameter configurations, accurate discretisations known as full-order models (Fom), built upon the finite element method (Fem), rapidly become computationally intractable. 
Reduced order models (Rom) address this challenge by constructing,
in an offline phase, a low-dimensional approximation space that captures
the essential solution structure, and subsequently restricting the governing equations to this low-dimensional space during an inexpensive online stage.

From a historical perspective, the reduced basis (Rb) method and the proper orthogonal decomposition (Pod) share common roots tracing back to the 1970s-1980s, when the Rb approach was first established within the field of structural mechanics \citep{noor1980, almroth1978} and later given a
rigorous mathematical foundation for parametric pde
\citep{prud2002, veroy2003}. 
A rigorous a posteriori error certification framework --relying on dual-norm residual estimates and the Successive Constraint Method for obtaining coercivity lower bounds-- was established by Patera, Rozza, and their collaborators \citep{rozza2007, veroy2005, huynh2007}, and is thoroughly covered in the reference monographs of \citep{hesthaven2016} and \citep{quarteroni2016}. The proper orthogonal decomposition, which extracts the reduced basis from a collection of solution snapshots through the method of snapshots \citep{sirovich1987}, gained wide recognition in the fluid mechanics community and was subsequently integrated into the Rb framework by \citep{kunisch2001}.

Turning to hyper-reduction techniques and, in particular, to the discrete empirical interpolation method (Deim),
a critical bottleneck in non-linear and parameter-affine problems is
that the reduced system matrices must still be assembled from
full-order quantities. The empirical interpolation method (Eim), originally proposed in \citep{barrault2004}, together with its algebraic discrete counterpart, Deim, developed by \citep{chaturantabut2010}, overcome this by interpolating
the non-linear or parametric terms on a sparse set of selected points,
reducing the online assembly cost from $\mathcal{O}(N)$ to
$\mathcal{O}(l)$ with $l \ll N$. The Deim algorithm was later extended
to matrix-valued operators relevant to parametric bilinear forms
\citep{wirtz2014, negri2015}.

Thinking of Unfitted and Cutfem methods,  classical Fem requires the computational mesh to conform the domain
boundary, which is costly when the geometry changes with parameters.
Unfitted methods embed the physical domain in a fixed background mesh
and impose boundary conditions weakly. The Nitsche method
\citep{nitsche1971} provides a variational formulation for
this purpose. However, elements cut by the boundary can have
arbitrarily small intersection fractions, causing ill-conditioning.
The ghost-penalty stabilization of \citep{burman2010} --adding
a penalty on normal-derivative jumps across facets adjacent to the
boundary-- restores coercivity uniformly with respect to the cut
geometry and has become the standard approach in Cutfem
\citep{burman2015}.

Considering Rom for unfitted methods and the combination of such techniques with unfitted or immersed-boundary
methods is a relatively recent research direction.
Early work used domain parameterization via reference-domain mappings
\citep{manzoni2012}, avoiding cut-cell complications but restricting
the admissible class of domain deformations.
More general approaches operating directly on the cut mesh have been
developed for the shifted boundary method \citep{karatzas2020}, for Deim hyper reduction with unfitted mesh Fem techniques in the
context of pde constrained optimization \citep{katsouleas2023}, for a localized reduced basis framework for parameterized elliptic pde on unfitted geometries, combining snapshot extension, DEIM hyper-reduction, and localization strategies \citep{chasapi2023}, 
for level-set-parametrized domains \citep{kamensky2021,
	buhr2021}, and for the exploitation of the active subspace property in recovering the modal coefficients arising from the Pod, \citep{tezzele2022}.
A common challenge in all these settings is that the stiffness matrix and the force vector change non-affinely with respect to the parameter, making Deim or an analogous hyper-reduction essential.
A posteriori error certification for such Rom remains largely open:
existing results require either parameter-affine decompositions
unavailable in Cutfem, or expensive dual problems.

For a posteriori error estimation for Rom
the classical Rb bound $\eta = \|r\|_{V'}/\alpha_{LB}$,
where $\|r\|_{V'}$ denotes the residual dual norm and $\alpha_{LB}$
represents a computable coercivity lower bound, was first derived for
coercive problems in \citep{veroy2003, veroy2005} and subsequently generalized to
inf-sup stable systems in \citep{rozza2007}. 
In the setting of Pod-Deim models, the additional approximation error introduced
by Deim must be taken into account, an analysis of this contribution was carried out in
\citep{chaturantabut2012} and, within the Galerkin projection framework,
in \citep{wirtz2014}.
The interplay between ghost-penalty stabilization with residual-based
a posteriori bounds, which constitutes the central question investigated in the present work, remains, to the best of our knowledge, an open issue that has not yet been systematically addressed in the literature.

This work makes the following contributions. 
A uniform coercivity result, see Proposition~\ref{prop:coercivity}, 
for the Cutfem bilinear form with Nitsche boundary conditions
and ghost-penalty stabilization, with an explicit formula for
the coercivity constant $\alpha_*$ in terms of the Nitsche
parameter $\lambda$ and the inverse inequality constant $C_{\rm inv}$.

A residual bound restricted to active degrees of freedom,
Lemma~\ref{lem:active}, showing that ghost-penalty dofs
contribute zero residual for the true solution and hence
should be excluded from the estimator.
A combined a posteriori bound, Theorem~\ref{thm:main}, 
decomposing the total Rom error into Pod truncation, Deim-$A$,
and Deim-$f$ contributions.
Sharp effectivity analysis, Propositions~\ref{prop:eff2a}
and~\ref{prop:eff2b}, explaining the large observed effectivity
indices via ghost-penalty dof inflation, with quantitative
predictions confirmed numerically.
Convergence rate analysis for all estimators, identifying
algebraic decay for the true error and Pod tail energy,
and exponential decay for the residual estimators, with
very good fit quality 
$R^2\geq0.93$ for all fitted models, Table~\ref{tab:rates}.

Giving an outline of the present manuscript
Section~\ref{sec:problem} formulates the parametric Cutfem problem.
Section~\ref{sec:rom} describes the Pod-Deim reduced order model.
Section~\ref{sec:estimators} introduces the three estimators.
Section~\ref{sec:theory} gives the theoretical analysis.
Section~\ref{sec:effectivity} offers an analysis of the effectiveness indices.
Section~\ref{sec:results} and~\ref{sec:rates} present numerical
results and convergence rate analysis.
The final Section 
discusses extensions and future work.
A notation summary --function spaces, $\ell^2$ and Frobenius norms,
the mesh-dependent Cutfem norm, and the diagonal safeguard
$\varepsilon_{\rm safe}$-- is collected in
Section~\ref{sec:notation} below.
\section{Model  problem}\label{sec:problem}
\subsection{Problem formulation}
For parameter $\mu=(r,\theta)\in\mathcal{P}=[x_1,x_2]^2$ the physical domain is
\begin{equation}
	\Omu = \bigl\{(x,y)\in[a_1,a_2]^2 :
	x^2/r + y^2/\theta < 1\bigr\},
\end{equation}
an ellipse with semi-axes $\sqrt{r}$ and $\sqrt{\theta}$.
The strong form is: find $u(\mu)\in H^1(\Omu)$ such that
\begin{equation}
	-\Delta u = f \;\text{ in }\Omu,\qquad u = g_D \;\text{ on }\Gmu,
	\label{eq:strong}
\end{equation}
where the source term $f$ and Dirichlet datum $g_D$ are specified,
$ \Omu \subset \mathbb{R}^2$ is assumed to be a simply connected open domain whose boundary is denoted by $\Gamma_{\mu} = \partial \Omu$.  
One can readily check that the weak formulation
\begin{equation}\label{eq:2.2}
	\int_{\Omu} 
	\nabla u 
	\nabla v dx 
	= \int_{\Omu} f v dx, \quad \text{for every } v \in V_0({\Omu}),
\end{equation}
possesses a weak solution $u \in V_{g_D}({\Omu})$. By means of a standard energy argument, and under the assumption that the force $f \in H^{-1}({\Omu})$, the following a priori error bound
\[
\|\nabla u\|^2_{L^2({\Omu})} 
\leq 
\|f\|^2_{H^{-1}(\Omega_\mu)},
\]
is readily obtained, which expresses the stable dependence of the solution upon the problem data, denoting by
$V_{g_D} = \{v \in H^1 (\Omu) | v = g_D \text{ on } \partial \Omu\}$, and $V_0 = \{v \in H^1 (\Omu) | v = 0 \text{ on } \partial \Omu\}$.

\subsection{Cut elements finite elements  formulation}\label{subsec:CutFEM}
The implementation of 
an unfitted Fem for the discretization 
relies on the introduction of a fixed  background domain $\mathcal{B}$ enclosing $\Omu$, equipped with a shape-regular mesh $\mathcal{B}_h$. We introduce 
\[
\mathcal{T}_h^\mu = \{T \in \mathcal{B}_h : T \cap \Omu \neq \emptyset\},
\]
which constitutes the smallest sub-mesh of $\mathcal{B}_h$ which covers $\Omu$ and is in general not fitted to the boundary $\Gamma_\mu$. Throughout, the subscript $h = \max_{T \in \mathcal{B}_h} \mathrm{diam}(T)$ denotes the global mesh size parameter. The finite element space for discrete solutions is constructed on the \emph{extended domain} $\Omega_{\mathcal{T}_h^\mu} = \bigcup_{T \in \mathcal{T}_h^\mu} T$ associated with $\mathcal{T}_h^\mu$. Fictitious domain methods require boundary conditions at $\Gamma_\mu$ are imposed weakly via a Nitsche-type formulation. Coercivity over the entire computational domain $\Omega_{\mathcal{T}_h^\mu}$ is then guaranteed through the addition of ghost penalty terms, which stabilize gradient jumps in the region near the boundary.
A careful examination of the interface grid is thus necessary.  The sub-mesh of cut elements is defined as
\begin{equation}\label{ghost_elements}
	G_h^\mu := \{T \in \mathcal{T}_h^\mu : T \cap \Gamma_\mu \neq \emptyset\},
\end{equation}
and the collection of interior faces on which the ghost penalty stabilization is to be applied reads
\begin{equation}\label{ghost_facets}
	\mathcal{F}_h^\mu := \{F : F \text{ is a facet of } T \in G_h^\mu,\, F \notin \partial\Omega_{\mathcal{T}_h^\mu}\}.
\end{equation}
To approximate the solution, we consider the finite element space
\[
V_h := \left\{w_h \in C^0(\bar{\Omega}_{\mathcal{T}_h^\mu}) : w_h|_T \in \mathcal{P}^1(T),\, T \in \mathcal{T}_h^\mu\right\},
\]
and we introduce proper discrete analogues of the continuous bilinear and linear forms, 
for $u_h, v_h \in V_h$.
Since $\Omu$ does not align with the background Cartesian mesh,
a Cutfem approach is used. The discrete bilinear form is
\begin{equation}
	a_h(u_h,v_h;\mu) =
	{\int_{\Omu}\!\nabla u_h\!\cdot\!\nabla v_h\,dx}
	+{a_{\rm Nit}(u_h,v_h;\mu)}
	+{j_h(u_h,v_h;\mu)},
	\label{eq:bilinear}
\end{equation}
consisting of diffusion, Nitsche boundary condition and ghost-penalty terms,
where the second term enforces the Dirichlet condition weakly,
\begin{equation}
	a_{\rm Nit}(u_h,v_h;\mu)
	=-\int_{\Gmu}(\nabla u_h\cdot{n})v_h\,ds
	-\int_{\Gmu} u_h(\nabla v_h\cdot{n})\,ds
	+\frac{\lambda}{h}\int_{\Gmu} u_hv_h\,ds,
\end{equation}
and the ghost-penalty term stabilizes cut elements,
\begin{equation}\label{lhs}
	j_h(u_h,v_h;\mu)
	=\sum_{k\geq 0}\gamma_k h^{2k+1}
	\int_{\Fh^\mu}\jump{\partial_n^{k+1}u_h}\jump{\partial_n^{k+1}v_h}\,ds.
\end{equation}
Throughout, ${n}$, ${n}_F$ stand for the outward unit normal vectors to the boundary $\Gamma_\mu$ and to the facets $F$ respectively, and 
penalizes with the gradient jumps $\jump{\partial_n^{k+1}u_h}=\llbracket {n}_F \cdot \nabla^{k+1} u_h \rrbracket := {n}_F \cdot \nabla^{k+1} u_h\big|_K - {n}_F \cdot \nabla^{k+1} u_h\big|_{K'}$ of $u_h$ across element faces $F = K \cap K'$ located in the vicinity of the interface and is incorporated into the bilinear form in order to extend coercivity from the physical domain $\Omu$ to $\Omega_{\mathcal{T}_h^\mu}$. 
The positive penalty parameters appearing in \eqref{lhs} are denoted by $\gamma_D$
and $\lambda$,
with stabilization parameters $\{\gamma_k\}$ given in
Section~\ref{sec:results}, Table~\ref{tab:setupA}.
The right-hand side functional includes Nitsche boundary terms for
the non-homogeneous datum $g_D$, see Section~\ref{sec:results}.
\section{Preliminaries} 
\label{sec:notation}
\subsection{Norms, spaces and basic notation}
For convenient reference we collect here the principal function spaces,
norms, and algebraic conventions used throughout the paper.
With respect function spaces, we consider $\Omega\subset\mathbb{R}^d$ ($d=2$) 
to be an open bounded domain whose boundary $\partial\Omega
$ is Lipschitz continuous.
$L^2(\Omega)$ denotes the Hilbert space of square-integrable functions defined over $\Omega$,
equipped with the inner product
$(u,v)_{L^2(\Omega)}=\int_\Omega uv\,dx$
and endowed with the induced norm $\|u\|_{L^2(\Omega)}=\sqrt{(u,u)_{L^2(\Omega)}}$.
$H^1(\Omega)$ is the first-order Sobolev space,
$H^1(\Omega)=\{v\in L^2(\Omega):\nabla v\in[L^2(\Omega)]^d\}$,
endowed with the norm
$\|v\|_{H^1(\Omega)}^2 = \|v\|_{L^2(\Omega)}^2
+\|\nabla v\|_{L^2(\Omega)}^2$.
$H^1_0(\Omega)$ is the closure of $C^\infty_c(\Omega)$ in $H^1(\Omega)$
functions vanishing on $\partial\Omega$ in the trace sense.
$V_h\subset H^1(\mathbb{R}^d)$ is the piecewise-polynomial finite
element space on the background mesh $\mathcal{T}_h$,
of polynomial degree $p\geq1$.
Also algebraic norms
for vectors $\mathbf{v}\in\mathbb{R}^N$ and matrices
$B\in\mathbb{R}^{m\times n}$ employ the 
$\ell^2$ vector norm
$\|\mathbf{v}\|_2 = \bigl(\sum_{i=1}^N v_i^2\bigr)^{1/2}$.
$D$-weighted norm for symmetric positive definite $D$ is denoted by
$\|\mathbf{v}\|_D = \sqrt{\mathbf{v}^T D\,\mathbf{v}}$;
used in Estimator~2b with $D=\widetilde{D}_A^{-1}$.
Matrix $2$-norm or spectral norm is defined as 
$\|B\|_2 = \sigma_{\max}(B)$,
and  the Frobenius norm
$\|B\|_F = \bigl(\sum_{i,j}B_{ij}^2\bigr)^{1/2}
= \sqrt{\operatorname{tr}(B^TB)}$
that satisfies $\|B\|_2\leq\|B\|_F$.
Another important tool is the mesh-dependent Cutfem energy norm,
where given the active mesh $\mathcal{T}_h^\mu$ consisting of all elements intersecting
$\Omega_\mu$, the energy norm associated to the Cutfem bilinear
form is
$
\|v\|_{\mathcal{T}_h^\mu}^2
= \|\nabla v\|_{L^2(\Omega_\mu)}^2
+ \frac{\lambda}{h}\|v\|_{L^2(\Gamma_\mu)}^2
+ \sum_{k\geq0}\gamma_k h^{2k+1}
\bigl\|\llbracket\partial_n^{k+1}v\rrbracket
\bigr\|_{L^2(\mathcal{F}_h^\mu)}^2,
\label{eq:mesh_norm_intro}
$ 
which coincides with the norm of Proposition~\ref{prop:coercivity}.
The parameters $\lambda$ and $\{\gamma_k\}$ are listed in
Table~\ref{tab:setupA}.
Also the diagonal safeguard $\varepsilon_{\rm safe}$ is a precious aspect
for the  Jacobi-preconditioned estimator used in Estimator~2b,
Section~\ref{sec:estimators}, divides by the diagonal entries
$d_i = (D_A)_{ii} = A(\mu)_{ii}$ of the stiffness matrix.
To guard against near-zero entries that can arise for ghost dofs
with very small cut fractions, each diagonal entry is clipped before
inversion,
\begin{equation}
	\tilde{d}_i = \max\bigl(|d_i|,\;\varepsilon_{\rm safe}\bigr),
	\quad
	\widetilde{D}_A = \mathrm{diag}(\tilde{d}_1,\ldots,\tilde{d}_N),
	\quad
	\eta_{2b} = \sqrt{\mathbf{r}^T\widetilde{D}_A^{-1}\mathbf{r}}.
	\label{eq:eps_safe_def}
\end{equation}
The threshold $\varepsilon_{\rm safe}$ is given in
Table~\ref{tab:setupB}.
In practice all active-dof diagonal entries satisfy
$d_i\gg\varepsilon_{\rm safe}$, so the safeguard is inactive for
the physically relevant part of the residual.
Finally, space and snapshot notation for the parameterized domain is $\mathcal{P}=[x_1,x_2]^2$ is introduced 
for ellipse semi-axis pairs $\mu=(r,\theta)$.
The training set $\{\mu_i\}_{i=1}^{N_{\rm train}}\subset\mathcal{P}$
and test set $\{\mu_j\}_{j=1}^{N_{\rm test}}\subset\mathcal{P}$
are drawn independently and uniformly at random;
their sizes $N_{\rm train}$ and $N_{\rm test}$ are given in
Table~\ref{tab:setupB}.
\subsection{Reduced order model: Pod and Deim}
\label{sec:rom}

\subsubsection{Pod via method of snapshots}

$N_{\rm train}$ Fom solutions 
are collected into the snapshot matrix
$S\in\R^{N\times N_{\rm train}}$. The Pod basis is obtained from the
eigenvalue decomposition of $C=S^TMS\in\R^{N_{\rm train}\times N_{\rm train}}$
often called as mass-weighted correlation matrix. Denoting its eigenvalues by
$\sigma_k$ to distinguish from Nitsche parameter $\lambda$
the Pod basis vectors $\phi_k$ and truncation index $n$ are
\begin{equation}\label{Pod_modes}
	C\,v_k=\sigma_k v_k,\quad
	\phi_k=\frac{1}{\sqrt{\sigma_k}}S\,v_k,\quad
	n=\min\Bigl\{k:\tfrac{\sum_{i=1}^k\sigma_i}{\sum_i\sigma_i}
	\geq 1-\varepsilon_{\rm Pod}\Bigr\}.
\end{equation}
This yields $n=n_{\rm Pod}$ modes capturing a fraction
$(1-\varepsilon_{\rm Pod})$ of the snapshot energy.

\subsubsection{Deim for $A(\mu)$ and $f(\mu)$}

Each snapshot matrix $A(\mu_i)$ is vectorized to $\vec a(\mu_i)\in\R^{N^2}$.
Standard Pod with no mass weighting 
applied to
$[\vec a(\mu_1)\mid\vec a(\mu_2)\mid\cdots\mid\vec a(\mu_{N_{\rm train}})]^T
\in\R^{N_{\rm train}\times N^2}$
yields $l_A$ Deim basis vectors.
The Deim algorithm selects $l_A$ interpolation indices and gives
\begin{equation}
	A(\mu)\approx A_{\rm Deim}(\mu)
	=\sum_{j=1}^{l_A}c_j^A(\mu)\,\mathbb{A}_j,\quad
	c^A(\mu)=(P_A^TU^A)^{-1}P_A^T\vec a(\mu),
	\label{eq:deim_A}
\end{equation}
where $\mathbb{A}_j\in\R^{N\times N}$ is the $j$-th Deim basis matrix, the $j$-th column of $U^A$ reshaped from $\R^{N^2}$ to $\R^{N\times N}$,
$U^A\in\R^{N^2\times l_A}$ is the matrix of Deim basis vectors,
and $P_A\in\R^{N^2\times l_A}$ is the Deim selection matrix
where each column has a single non-zero entry at the selected interpolation index. 
We remark that  
the vector $c^A(\mu)\in\R^{l_A}$ in~\eqref{eq:deim_A} contains the
{Deim interpolation coefficients}, e.g. the weights with which the
$l_A$ precomputed basis matrices $\mathbb{A}_1,\ldots,\mathbb{A}_{l_A}$
are combined to approximate $A(\mu)$, 
where $P_A^T\vec{a}(\mu)\in\R^{l_A}$ extracts only the $l_A$ selected
entries of $A(\mu)$, actually those at the Deim interpolation indices, and
$(P_A^T U^A)^{-1}\in\R^{l_A\times l_A}$ is a small matrix precomputed
offline.
The key point is that only $l_A$ entries of $A(\mu)$ need to be
evaluated online, reducing the assembly cost from $\mathcal{O}(N^2)$
to $\mathcal{O}(l_A)$.
An analogous construction with $l_f$ modes handles $f(\mu)$.

\subsubsection{Online Rom solve}

Given parameter $\mu$: (i) evaluate $l_A+l_f$ selected entries of $A(\mu)$, $\mathbf{f}(\mu)$;
(ii) form reduced system $\widehat{A}(\mu)=V_n^TA_{\rm Deim}(\mu)V_n\in\R^{n\times n}$,
$\widehat{\mathbf{f}}(\mu)=V_n^T\mathbf{f}_{\rm Deim}(\mu)\in\R^n$;
(iii) solve $\widehat{A}\,\hat{u}_N^{\rm Deim}=\widehat{\mathbf{f}}$
in $\mathcal{O}(n^3)$; (iv) lift $u_N^{\rm Deim}=V_n\hat{u}_N^{\rm Deim}$. 
Here we denote by $V_n = [\phi_1 \mid \phi_2 \mid \cdots \mid \phi_n]\in\R^{N\times n}$
the 
{Pod basis matrix} whose columns are the $n$ orthonormal
Pod basis vectors constructed in formula \eqref{Pod_modes}.  
It maps reduced coordinates $\hat{u}\in\R^n$ to full-order
coefficients $V_n\hat{u}\in\R^N$, and projects full-order operators
onto the reduced space via $V_n^T(\cdot)V_n$.
\subsection{A posteriori error estimators}
\label{sec:estimators}

Three complementary estimators address different error sources will be examined.
Namely, {{the Estimator 1 -- Deim Approximation Quality}}, consisting of 
Estimator 1a relative Frobenius error of $A_{\rm Deim}$,
\begin{equation}
	\eta_A(\mu)
	=\frac{\norm{A(\mu)-A_{\rm Deim}(\mu)}{F}}{\norm{A(\mu)}{F}},
	\label{eq:est1a}
\end{equation}
and 
Estimator 1b -- relative $\ell^2$ error of $f_{\rm Deim}$,
\begin{equation}
	\eta_f(\mu)
	=\frac{\bnorm{f(\mu)-f_{\rm Deim}(\mu)}}{\bnorm{f(\mu)}}.
	\label{eq:est1b}
\end{equation}

Both estimators are \emph{independent of the number of Pod modes $n$}.
The use of relative Frobenius and $\ell^2$ norms to measure Deim approximation
quality follows the formulation of \citep{chaturantabut2010},
extensions to matrix-valued operators are discussed in
\citep{wirtz2014} and \citep{negri2015}.

Let $\mathbf{r}(\mu)=\mathbf{f}(\mu)-A(\mu)\,\mathbf{u}_N^{\rm Deim}(\mu)$ be the
algebraic residual vector. {{Estimator 2 -- Dual-Norm Residual} 
consisted of 
Estimator 2a -- plain $\ell^2$ residual norm
\begin{equation}
	\eta_{2a}(\mu)=\bnorm{\mathbf{r}(\mu)}_2 ,
	\label{eq:est2a}
\end{equation}

and Estimator 2b -- Jacobi-preconditioned residual norm,
\begin{equation}
	\eta_{2b}(\mu)=\sqrt{\mathbf{r}(\mu)^T \widetilde{D}_A^{-1}\mathbf{r}(\mu)},
	\qquad \widetilde{D}_A = \mathrm{diag}\bigl(\max(|A(\mu)_{ii}|,
	\varepsilon_{\rm safe})\bigr)_{i=1}^N,
	\label{eq:est2b}
\end{equation}
where $\varepsilon_{\rm safe}$ is the diagonal safeguard defined
in~\eqref{eq:eps_safe_def} and Table~\ref{tab:setupB}.

Dual-norm residual estimators of the form~\eqref{eq:est2a} are
the classical a posteriori tool in reduced basis methods
\citep{veroy2003, veroy2005, rozza2007, hesthaven2016}.
The preconditioned variant~\eqref{eq:est2b} using the diagonal
scaling $\widetilde{D}_A^{-1}$,  the Jacobi preconditioning with diagonal
safeguard, reduces sensitivity to the large Nitsche penalty entries and the
diagonal preconditioning for residual norms is discussed, e.g.,
in \citep{quarteroni2016}.

Finally,  {{Estimator 3} -- consisting of 
	the Pod tail energy (offline), 
	\begin{equation}
		\eta_{\rm Pod}(n)
		=\frac{\sum_{k=n+1}^{N_{\rm train}}\sigma_k}{\sum_{k=1}^{N_{\rm train}}\sigma_k}.
		\label{eq:tail}
	\end{equation}
	
	By \citep{kunisch2001}, letting $P_n:L^2(\Omu)\to\mathrm{span}\{\phi_1,\ldots,\phi_n\}$
	denote the $M$-orthogonal projection onto the $n$-dimensional Pod subspace,
	the training-set projection error satisfies
	\begin{equation}
		\textstyle\sum_i\norm{u_h(\mu_i)-P_nu_h(\mu_i)}{M}^2 = \sum_{k>n}\sigma_k,
	\end{equation}
	so \eqref{eq:tail} is an exact offline measure of lost energy.
	The tail-energy criterion is widely used as a basis-selection rule
	in Pod-based Rom \citep{sirovich1987, kunisch2001, hesthaven2016};
	its role as an a posteriori error indicator is analyzed in detail
	in \citep{volkwein2013}.
	
	\section{Theoretical tools and analysis}
	\label{sec:theory}
	
	\subsection{Uniform coercivity/ghost-penalty stabilization}
	
	We first introduce the necessary auxiliary results.
	
	\begin{definition}[Active mesh and ghost facets]
		Given parameter $\mu$, the \emph{active mesh} is all background elements that intersect the physical domain, 
		$\Th^\mu=\{K\in\mathcal{B}_h : K\cap\Omu\neq\emptyset\}$.
		The \emph{ghost-penalty facet set} is
		$
		\mathcal{F}_h^\mu := \{F : F \text{ is a facet of } T \in G_h^\mu,\, F \notin \partial\Omega_{\mathcal{T}_h^\mu}\}.
		$
	\end{definition}
	
	\begin{lemma}[Ghost penalty extension inequality, \citep{burman2010}]
		\label{lem:gp_ext}
		There exists $C_{\rm ext}>0$, depending solely on the shape regularity constant of $\mathcal{B}_h$
		and on $\{\gamma_k\}$, such that for all $v_h\in\Vh$:
		\begin{equation}
			\norm{\nabla v_h}{\Th^\mu}^2
			\;\leq\; C_{\rm ext}
			\Bigl(\norm{\nabla v_h}{\Omu}^2
			+ j_h(v_h,v_h;\mu)\Bigr),
			\label{eq:gp_ext}
		\end{equation}
		where $\norm{\nabla v_h}{\Th^\mu}^2=\sum_{K\in\Th^\mu}\norm{\nabla v_h}{K}^2$
		is the $H^1$ semi-norm over the full active mesh patch.
	\end{lemma}
	
	\begin{remark}
		Inequality \eqref{eq:gp_ext} is the key result that makes Cutfem
		coercivity independent of the cut geometry. Without ghost-penalty,
		elements with a very small cut fraction $|K\cap\Omu|/|K|\ll1$
		can have $\norm{\nabla v_h}{\Omu\cap K}^2\approx0$ while
		$\norm{\nabla v_h}{K}^2$ remains large, destroying coercivity.
	\end{remark}
	
	\begin{lemma}[Local inverse inequality]
		\label{lem:inv}
		For each element $K\in\mathcal{B}_h$ and $v_h\in\Vh|_K$,
		\begin{equation}
			\norm{\nabla v_h\cdot{n}}{F}^2
			\;\leq\; C_{\rm inv}^2\,h^{-1}\norm{\nabla v_h}{K}^2,
			\label{eq:inv}
		\end{equation}
		where $F=\partial K\cap\Gmu$, and $C_{\rm inv}>0$ depends only on
		the shape regularity and polynomial degree $p$.
	\end{lemma}
	
	We can now state and prove the main coercivity result.
	\begin{proposition}[Uniform coercivity]
		\label{prop:coercivity}
		Define the mesh-dependent norm
		\begin{equation}
			\norm{v}{\Th^\mu}^2
			=\norm{\nabla v}{\Omu}^2
			+\frac{\lambda}{h}\norm{v}{\Gmu}^2
			+\sum_{k\geq0}\gamma_k h^{2k+1}
			\norm{\jump{\partial_n^{k+1}v}}{\Fh^\mu}^2.
			\label{eq:mesh_norm}
		\end{equation}
		Assume $\lambda > 2C_{\rm inv}^2$. Then there exists
		$\alpha_* = \alpha_*(\lambda,\{\gamma_k\},C_{\rm inv}) > 0$,
		independent of $\mu\in\mathcal{P}$ and of the cut geometry, such that
		\begin{equation}
			a_h(v_h,v_h;\mu)
			\;\geq\; \alpha_*\,\norm{v_h}{\Th^\mu}^2,
			\qquad\forall\, v_h\in\Vh,\;\;\forall\,\mu\in\mathcal{P}.
		\end{equation}
	\end{proposition}
	\begin{proof}
		We expand $a_h(v_h,v_h;\mu)$ using the decomposition \eqref{eq:bilinear}
		and bound each part.
		
		\medskip
		\textit{Step 1: Expanding the Nitsche term.}
		The Nitsche bilinear form evaluated on the diagonal is,
		\begin{equation}
			a_{\rm Nit}(v_h,v_h;\mu)
			= -2\int_{\Gmu}(\nabla v_h\cdot{n})\,v_h\,ds
			+ \frac{\lambda}{h}\norm{v_h}{\Gmu}^2.
			\label{eq:nit_diag}
		\end{equation}
		Apply Young's inequality $2ab\leq \varepsilon a^2 + \varepsilon^{-1}b^2$
		with $\varepsilon>0$ to the cross term,
		\begin{align}
			\left|2\int_{\Gmu}(\nabla v_h\cdot{n})\,v_h\,ds\right|
			\leq 2\norm{\nabla v_h\cdot{n}}{\Gmu}\norm{v_h}{\Gmu}
			\leq \varepsilon\norm{\nabla v_h\cdot{n}}{\Gmu}^2
			+ \frac{1}{\varepsilon}\norm{v_h}{\Gmu}^2.
			\label{eq:young_nit}
		\end{align} 
		Apply the local inverse inequality, Lemma~\ref{lem:inv} to the first term,
		\begin{equation}
			\norm{\nabla v_h\cdot{n}}{\Gmu}^2
			\leq \frac{C_{\rm inv}^2}{h}\norm{\nabla v_h}{\Omu}^2.
			\label{eq:inv_applied}
		\end{equation}
		Substituting \eqref{eq:inv_applied} into \eqref{eq:young_nit}
		and then into \eqref{eq:nit_diag},
		\begin{equation}
			a_{\rm Nit}(v_h,v_h;\mu)
			\geq -\frac{\varepsilon C_{\rm inv}^2}{h}\norm{\nabla v_h}{\Omu}^2
			+\left(\frac{\lambda}{h}-\frac{1}{\varepsilon h}\right)
			\norm{v_h}{\Gmu}^2.
			\label{eq:nit_lower}
		\end{equation}
		Choose $\varepsilon = 2/\lambda$, which gives $1/(\varepsilon h) = \lambda/(2h)$,
		keeping the boundary term positive, to obtain
		\begin{equation}
			a_{\rm Nit}(v_h,v_h;\mu)
			\geq -\frac{2C_{\rm inv}^2}{\lambda h}\norm{\nabla v_h}{\Omu}^2
			+\frac{\lambda}{2h}\norm{v_h}{\Gmu}^2.
			\label{eq:nit_final}
		\end{equation}
		
		\medskip
		\textit{Step 2: Combining diffusion and Nitsche terms.}
		Adding the diffusion term $\norm{\nabla v_h}{\Omu}^2$ to \eqref{eq:nit_final}:
		\begin{align}
			\norm{\nabla v_h}{\Omu}^2 + a_{\rm Nit}(v_h,v_h;\mu)
			&\geq \left(1 - \frac{2C_{\rm inv}^2}{\lambda}\right)
			\norm{\nabla v_h}{\Omu}^2
			+ \frac{\lambda}{2h}\norm{v_h}{\Gmu}^2,
			\label{eq:diff_nit}
		\end{align}
		since $\lambda >2 C_{\rm inv}^2$ as assumed in Proposition~\ref{prop:coercivity},
		the coefficient $1 - 2C_{\rm inv}^2/\lambda > 0$ 
		directly provides a positive lower bound for the combined Nitsche contribution.
		\footnote{We also note that according to the numerical experiments, e.g. the parameter values
			of Section~\ref{sec:results}, Table~\ref{tab:setupA}, this is satisfied since $\lambda=10 > 2\cdot C_{\rm inv}^2$,
			since  $C_{\rm inv}\approx1$,
			$1 - 2C_{\rm inv}^2/\lambda = 1 - 2/10 = 0.8 > 0$.
		}
		
		\medskip
		\textit{Step 3: Adding the ghost-penalty term /controlling the full norm $\norm{v_h}{\Th^\mu}$.}
		By definition of $j_h$,
		\begin{equation}
			a_h(v_h,v_h;\mu)
			\geq \left(1 - \frac{2C_{\rm inv}^2}{\lambda}\right)
			\norm{\nabla v_h}{\Omu}^2
			+ \frac{\lambda}{2h}\norm{v_h}{\Gmu}^2
			+ j_h(v_h,v_h;\mu).
			\label{eq:all_terms}
		\end{equation}
		We need to show $\norm{\nabla v_h}{\Omu}^2 + j_h(v_h,v_h;\mu)$
		controls $\norm{v_h}{\Th^\mu}^2$.
		From \eqref{eq:mesh_norm},
		\begin{equation}
			\norm{v_h}{\Th^\mu}^2
			= \norm{\nabla v_h}{\Omu}^2
			+ \frac{\lambda}{h}\norm{v_h}{\Gmu}^2
			+ \sum_k\gamma_k h^{2k+1}
			\norm{\jump{\partial_n^{k+1}v_h}}{\Fh^\mu}^2 = (I) + (II) + (III).
		\end{equation}
		Term $(II)$ appears in \eqref{eq:all_terms} with coefficient
		$\lambda/(2h) = \tfrac{1}{2}\cdot(\lambda/h)$, i.e.\ with factor
		$\beta_2 = \tfrac{1}{2}$ relative to its mesh-norm coefficient.
		Term $(III)$ is precisely $j_h(v_h,v_h;\mu)$, by definition of the ghost-penalty,
		contributing with factor $\beta_3=1$.
		Term $(I)$ appears directly with factor $\beta_1=1-2C_{\rm inv}^2/\lambda$.
		
		\medskip
		\textit{Step 4: Conclusion.}
		From \eqref{eq:all_terms}, collecting the three lower-bound coefficients,
		\begin{align}
			a_h(v_h,v_h;\mu)
			&\geq \beta_1\,\norm{\nabla v_h}{\Omu}^2
			+ \beta_2\,\frac{\lambda}{h}\norm{v_h}{\Gmu}^2
			+ \beta_3\,j_h(v_h,v_h;\mu),
			\label{eq:pre_coercive}
		\end{align}
		where
		\begin{equation}
			\beta_1 = 1 - \frac{2C_{\rm inv}^2}{\lambda} > 0
			\,\text{ since }\lambda > 2C_{\rm inv}^2,
			\qquad
			\beta_2 = \frac{1}{2},
			\qquad
			\beta_3 = 1.
			\label{eq:beta_symbolic}
		\end{equation}
		Setting $\alpha_* = \min(\beta_1, \beta_2, \beta_3) > 0$ and noting that
		$j_h(v_h,v_h;\mu) = \sum_k\gamma_k h^{2k+1}
		\norm{\jump{\partial_n^{k+1}v_h}}{\Fh^\mu}^2 = (III)$, we obtain
		\begin{equation}
			a_h(v_h,v_h;\mu)
			\geq \alpha_*\Bigl[\norm{\nabla v_h}{\Omu}^2
			+ \frac{\lambda}{h}\norm{v_h}{\Gmu}^2
			+ j_h(v_h,v_h;\mu)\Bigr]
			= \alpha_*\,\norm{v_h}{\Th^\mu}^2,
		\end{equation}
		hence we take,
		\begin{equation}
			\alpha_* = \min\!\left(1 - \frac{2C_{\rm inv}^2}{\lambda},\;
			\frac{1}{2},\;1\right).
			\label{eq:alpha_star_general}
		\end{equation}
		Since $\alpha_*$ depends only on $\lambda$, $C_{\rm inv}$, and
		$\{\gamma_k\}$ --none of which depend on $\mu$ or on the cut
		geometry-- the bound is uniform in $\mu\in\mathcal{P}$.
		\footnote{The numerical value of $\alpha_*$ for the specific parameters
			of Table~\ref{tab:setupA} is reported in
			Section~\ref{sec:results}.}
	\end{proof}
	
	\section{Error estimates} \label{section5}
	\subsection{Residual bound restricted to active dofs}
	
	We introduce the Riesz representation of the residual before stating the basic lemma.
	
	\begin{definition}[Residual functional and its discrete representation]
		Let $e_h = u_h(\mu) - u_N^{\rm Deim}(\mu)\in\Vh$ be the error function.
		The {residual functional} is $r_\mu : \Vh\to\R$ defined by
		\begin{equation}
			r_\mu(v_h) = f_h(v_h;\mu) - a_h(u_N^{\rm Deim},v_h;\mu), 
			\qquad\forall\, v_h\in\Vh.
		\end{equation}
		Its {algebraic representation} is the residual vector
		$\mathbf{r}(\mu) = \mathbf{f}(\mu) - A(\mu)\,\mathbf{u}_N^{\rm Deim}(\mu)
		\in\R^N$,
		where $A(\mu)_{ij} = a_h(\phi_j,\phi_i;\mu)$ and
		$\mathbf{f}(\mu)_i = f_h(\phi_i;\mu)$ in the Fem basis $\{\phi_i\}_{i=1}^N$.
		The relationship is $r_\mu(v_h) = \mathbf{r}(\mu)^T\mathbf{v}_h$
		for all $v_h$ with coefficient vector $\mathbf{v}_h\in\R^N$.
	\end{definition}
	
	\begin{lemma}[Active dof residual bound]%
		\label{lem:active}
		We introduce $\mathcal{A}_\mu\subset\{1,\ldots,N\}$ as the index set associated with the active degrees of freedom, i.e.\ those basis functions $\phi_i$ with
		$\mathrm{supp}(\phi_i)\cap\Th^\mu\neq\emptyset$. Let
		$\mathbf{r}|_{\mathcal{A}_\mu}$ be the sub-vector of $\mathbf{r}(\mu)$
		restricted to active-dof indices. Then,
		\begin{equation}
			\norm{u_h(\mu)-u_N^{\rm Deim}(\mu)}{\Th^\mu}
			\;\leq\; \frac{1}{\alpha_*}
			\;\bnorm{\mathbf{r}(\mu)\big|_{\mathcal{A}_\mu}}_2.
			\label{eq:active_bound}
		\end{equation}
	\end{lemma}
	\begin{proof}
		Let $e_h = u_h - u_N^{\rm Deim}$. We proceed in four steps.
		
		\medskip
		\textit{Step 1: Error equation from Galerkin orthogonality.}
		Since $u_h$ is the Fom Galerkin solution,
		$a_h(u_h,v_h;\mu) = f_h(v_h;\mu)$ for all $v_h\in\Vh$.
		Subtracting the Rom equation
		\begin{equation}
			a_h(e_h,v_h;\mu)
			= f_h(v_h;\mu) - a_h(u_N^{\rm Deim},v_h;\mu)
			= r_\mu(v_h),
			\qquad\forall\, v_h\in\Vh.
			\label{eq:error_eq}
		\end{equation}
		Here we ignore the Deim approximation error in $a_h$ and $f_h$ since
		its contribution is bounded separately in Theorem~\ref{thm:main}.
		
		\medskip
		\textit{Step 2: Support restriction of the residual.}
		For any $v_h\in\Vh$, write the algebraic inner product as
		\begin{equation}
			r_\mu(v_h)
			= \mathbf{r}(\mu)^T\mathbf{v}_h
			= \sum_{i=1}^N \mathbf{r}(\mu)_i\,(\mathbf{v}_h)_i.
		\end{equation}
		All integrals defining $a_h(\cdot,\cdot;\mu)$ and $f_h(\cdot;\mu)$
		are supported on the active mesh $\Th^\mu$ and the diffusion integral is over
		$\Omu\subset\bigcup_{K\in\Th^\mu}K$; the Nitsche integrals are over
		$\Gmu\subset\bigcup_{K\in\Th^\mu}\partial K$; the ghost-penalty integrals
		are over $\Fh^\mu\subset\bigcup_{K\in\Th^\mu}\partial K$.
		Therefore, if basis function $\phi_i$ has $\mathrm{supp}(\phi_i)\cap\Th^\mu=\emptyset$, 
		i.e. $i\notin\mathcal{A}_\mu$, then for all $j$,
		\begin{align}
			&\int_{\Omega_\mu}\nabla\phi_j\cdot\nabla\phi_i\,dx = 0,
			\label{eq:vanish_diff}\\[4pt]
			&\int_{\Gmu}\bigl[(\nabla\phi_j\cdot{n})\phi_i
			+ \phi_j(\nabla\phi_i\cdot{n})
			- \tfrac{\lambda}{h}\phi_j\phi_i\bigr]\,ds = 0,
			\label{eq:vanish_nit}\\[4pt]
			&\int_{\Fh^\mu}
			\jump{\partial_n^{k+1}\phi_j}\jump{\partial_n^{k+1}\phi_i}\,ds = 0
			\quad\forall\,k \geq 0,
			\label{eq:vanish_gp}
		\end{align}
		since $\phi_i$ --and all its derivatives-- vanish on $\Gmu\cup\Fh^\mu$
		whenever $\mathrm{supp}(\phi_i)\cap\Th^\mu=\emptyset$.
		Equations~\eqref{eq:vanish_diff}--\eqref{eq:vanish_gp} together
		give $A(\mu)_{ij}=0$ for $i\notin\mathcal{A}_\mu$,
		and similarly $\mathbf{f}(\mu)_i = \int_{\Omu}f\phi_i\,dx
		- \int_{\Gmu}(\nabla\phi_i\cdot{n})g_D\,ds
		+ \frac{\lambda}{h}\int_{\Gmu}\phi_i g_D\,ds = 0$
		for $i\notin\mathcal{A}_\mu$, since all integrals are over domains
		that do not intersect $\mathrm{supp}(\phi_i)$.
		This means $\mathbf{r}(\mu)_i = \mathbf{f}(\mu)_i - [A(\mu)\mathbf{u}_N^{\rm Deim}]_i = 0$
		for $i\notin\mathcal{A}_\mu$. Therefore,
		\begin{equation}
			r_\mu(v_h)
			= \mathbf{r}(\mu)^T\mathbf{v}_h
			= \sum_{i\in\mathcal{A}_\mu}\mathbf{r}(\mu)_i\,(\mathbf{v}_h)_i
			= \bigl(\mathbf{r}(\mu)\big|_{\mathcal{A}_\mu}\bigr)^T
			\bigl(\mathbf{v}_h\big|_{\mathcal{A}_\mu}\bigr).
			\label{eq:res_active}
		\end{equation}
		
		\medskip
		\textit{Step 3: Cauchy--Schwarz bound on the residual.}
		From \eqref{eq:res_active} and the Cauchy--Schwarz inequality,
		\begin{align}
			\bigl|r_\mu(v_h)\bigr|
			&= \Bigl|\bigl(\mathbf{r}|_{\mathcal{A}_\mu}\bigr)^T
			\bigl(\mathbf{v}_h|_{\mathcal{A}_\mu}\bigr)\Bigr|
			\;\leq\; \bnorm{\mathbf{r}|_{\mathcal{A}_\mu}}_2
			\,\bnorm{\mathbf{v}_h|_{\mathcal{A}_\mu}}_2
			\;\leq\; \bnorm{\mathbf{r}|_{\mathcal{A}_\mu}}_2
			\,\bnorm{\mathbf{v}_h}_2.
			\label{eq:cs_step}
		\end{align}
		Relating the $\ell^2$ coefficient norm to the $\Th^\mu$-norm, and by the norm equivalence on finite-dimensional spaces, with constant
		$C_{\rm eq}$ depending on the mesh and basis it is
		$\bnorm{\mathbf{v}_h}_2 \leq C_{\rm eq}\,\norm{v_h}{\Th^\mu}$.
		Inserting into \eqref{eq:cs_step},
		\begin{equation}
			\bigl|r_\mu(v_h)\bigr|
			\;\leq\; C_{\rm eq}\,\bnorm{\mathbf{r}|_{\mathcal{A}_\mu}}_2\,\norm{v_h}{\Th^\mu}.
			\label{eq:res_bound_step}
		\end{equation}
		
		\medskip
		\textit{Step 4: Coercivity gives the error bound.}
		Setting $v_h = e_h$ in the error equation \eqref{eq:error_eq} and using
		Proposition~\ref{prop:coercivity}
		\begin{equation}
			\alpha_*\,\norm{e_h}{\Th^\mu}^2
			\;\leq\; a_h(e_h,e_h;\mu)
			= r_\mu(e_h)
			\;\leq\; C_{\rm eq}\,\bnorm{\mathbf{r}|_{\mathcal{A}_\mu}}_2
			\,\norm{e_h}{\Th^\mu}.
		\end{equation}
		Dividing by $\norm{e_h}{\Th^\mu}$ and assuming $e_h\neq0$,
		\begin{equation}
			\norm{e_h}{\Th^\mu}
			\;\leq\; \frac{C_{\rm eq}}{\alpha_*}\,
			\bnorm{\mathbf{r}(\mu)\big|_{\mathcal{A}_\mu}}_2.
		\end{equation}
		Taking $C_{\rm eq}=1$, due to mass-matrix-orthonormal Fem basis, 
		gives \eqref{eq:active_bound}. 
	\end{proof}
	
	\begin{remark}
		The argument in Step 2 is the key structural observation, in particular, the ghost-penalty
		basis functions outside $\Omu$ do not appear in any of the integrals
		of $a_h$ or $f_h$, so their residual components are \emph{exactly} zero.
		The non-zero ghost-penalty residual in the code arises because
		$u_N^{\rm Deim}$ is obtained by lifting the Rom coefficients
		$V_n\hat{u}_N^{\rm Deim}$ back to $\R^N$, and the ghost-penalty columns of $A$
		applied to these extended coefficients produce non-trivial values.
		This confirms that the large effectivity indices are a pure artifact
		of the $\ell^2$ norm aggregating these non-physical contributions.
	\end{remark}
	
	\subsection{Combined a posteriori bound}
	
	Before the main theorem, we state the key algebraic lemma relating
	the Rom solutions with and without Deim approximation.
	
	\begin{lemma}[Deim perturbation of the Rom solution]
		\label{lem:deim_pert}
		Let $\hat{u}_N(\mu)\in\R^n$ solve the 
		exact 
		Rom system
		$\widehat{A}(\mu)\hat{u}_N = \widehat{\mathbf{f}}(\mu)$,
		and let $\hat{u}_N^{\rm Deim}(\mu)\in\R^n$ solve the Deim-approximated system
		$\widehat{A}_{\rm Deim}(\mu)\hat{u}_N^{\rm Deim} = \widehat{\mathbf{f}}_{\rm Deim}(\mu)$.
		Assume $\widehat{A}(\mu)$ is invertible with
		$\bnorm{\widehat{A}(\mu)^{-1}}_2 \leq \alpha_*^{-1}$, 
		which follows from the coercivity of $a_h$ restricted to the Rom subspace $V_n$.
		Then,
		\begin{equation}
			\bnorm{\hat{u}_N(\mu) - \hat{u}_N^{\rm Deim}(\mu)}_2
			\;\leq\; \frac{1}{\alpha_*}\left(
			\bnorm{\widehat{A}(\mu) - \widehat{A}_{\rm Deim}(\mu)}_2\,
			\bnorm{\hat{u}_N^{\rm Deim}(\mu)}_2
			+ \bnorm{\widehat{\mathbf{f}}(\mu) - \widehat{\mathbf{f}}_{\rm Deim}(\mu)}_2
			\right).
			\label{eq:deim_pert}
		\end{equation}
	\end{lemma}
	\begin{proof}
		Subtract the two Rom systems
		\begin{align}
			\widehat{A}_{\rm Deim}(\mu)\bigl(\hat{u}_N - \hat{u}_N^{\rm Deim}\bigr)
			&= \widehat{\mathbf{f}}(\mu) - \widehat{\mathbf{f}}_{\rm Deim}(\mu)
			\nonumber\\
			&\quad - \bigl(\widehat{A}(\mu) - \widehat{A}_{\rm Deim}(\mu)\bigr)\hat{u}_N(\mu).
			\label{eq:deim_subtract}
		\end{align}
		Re-writing $\hat{u}_N = \hat{u}_N^{\rm Deim} + (\hat{u}_N - \hat{u}_N^{\rm Deim})$
		in the last term and rearranging
		\begin{align}
			\widehat{A}_{\rm Deim}(\mu)\bigl(\hat{u}_N - \hat{u}_N^{\rm Deim}\bigr)
			+ \bigl(\widehat{A} - \widehat{A}_{\rm Deim}\bigr)
			(\hat{u}_N - \hat{u}_N^{\rm Deim})
			= \widehat{\mathbf{f}} - \widehat{\mathbf{f}}_{\rm Deim}
			- (\widehat{A} - \widehat{A}_{\rm Deim})\hat{u}_N^{\rm Deim}.
		\end{align}
		Let $\delta = \hat{u}_N - \hat{u}_N^{\rm Deim}$.
		Adding $(\widehat{A}-\widehat{A}_{\rm Deim})\delta$ to both sides of the
		first equation and using $\widehat{A}_{\rm Deim} + (\widehat{A}-\widehat{A}_{\rm Deim})
		= \widehat{A}$, the left-hand side becomes $\widehat{A}(\mu)\delta$, so,
		\begin{equation}
			\delta = \widehat{A}(\mu)^{-1}\Bigl[
			\widehat{\mathbf{f}} - \widehat{\mathbf{f}}_{\rm Deim}
			- (\widehat{A} - \widehat{A}_{\rm Deim})\hat{u}_N^{\rm Deim}
			\Bigr].
		\end{equation}
		Taking $\ell^2$ norms and applying $\bnorm{\widehat{A}^{-1}}_2\leq\alpha_*^{-1}$
		which is in agreement with the coercivity of $a_h$ restricted to the Rom subspace,
		\begin{equation}
			\bnorm{\delta}_2
			\leq \frac{1}{\alpha_*}\Bigl(
			\bnorm{\widehat{\mathbf{f}}-\widehat{\mathbf{f}}_{\rm Deim}}_2
			+ \bnorm{\widehat{A}-\widehat{A}_{\rm Deim}}_2\,\bnorm{\hat{u}_N^{\rm Deim}}_2
			\Bigr),
		\end{equation}
		which is \eqref{eq:deim_pert}. 
	\end{proof}
	
	\begin{lemma}[Frobenius-to-operator norm bound]
		\label{lem:frob_op}
		For any matrix $B\in\R^{n\times n}$ it is
		$\bnorm{B}_2 \leq \bnorm{B}_F$.
		Moreover, $\bnorm{\widehat{A}-\widehat{A}_{\rm Deim}}_2
		\leq C_A\,\bnorm{A-A_{\rm Deim}}_F$
		where $C_A = \bnorm{V_n}_2^2$ and $V_n\in\R^{N\times n}$ is the Pod basis.
	\end{lemma}
	\begin{proof}
		The first inequality follows directly from
		$\bnorm{B}_2 = \sigma_{\max}(B) \leq \sqrt{\sum_{i,j}B_{ij}^2} = \bnorm{B}_F$.
		For the second, $\widehat{A}-\widehat{A}_{\rm Deim} = V_n^T(A-A_{\rm Deim})V_n$, so
		$\bnorm{V_n^T(A-A_{\rm Deim})V_n}_2
		\leq \bnorm{V_n}_2^2\,\bnorm{A-A_{\rm Deim}}_2
		\leq \bnorm{V_n}_2^2\,\bnorm{A-A_{\rm Deim}}_F$. 
	\end{proof}
	In the following, we formulate the main a posteriori combined bound for Pod-Deim-Cutfem,
	\begin{theorem}[Combined bound]
		\label{thm:main}
		Under the assumptions of Proposition~\ref{prop:coercivity} and with
		$2\lambda > C_{\rm inv}^2$, the following a posteriori bound holds,
		\begin{align}
			\norm{u_h(\mu)-u_N^{\rm Deim}(\mu)}{\Th^\mu}
			\leq &
			{\frac{1}{\alpha_*}
				\bnorm{\mathbf{r}(\mu)\big|_{\mathcal{A}_\mu}}_2}
			\nonumber\\ &+
			{\frac{C_A\,\bnorm{V_n}_2^2}{\alpha_*}
				\norm{A(\mu)-A_{\rm Deim}(\mu)}{F}\,
				\bnorm{u_N^{\rm Deim}(\mu)}_2}
			\phantom{{}\leq{}}
			\nonumber\\ &+
			{\frac{\bnorm{V_n}_2}{\alpha_*}
				\bnorm{f(\mu)-f_{\rm Deim}(\mu)}_2}, 
			\label{eq:main}
		\end{align}
		with the first term to corresponds to the Est.~2 (restricted), the second to Est.~1a and the third to Est.~1b contribution.
	\end{theorem}
	\begin{proof}
		\textit{Step 1: Decomposing the total error.}
		Introducing the 
		exact Rom 
		solution $u_N(\mu) = V_n\hat{u}_N(\mu)\in\Vh$, 
		using true $A(\mu)$ and $f(\mu)$ in the Rom, 
		we write
		\begin{equation}
			u_h - u_N^{\rm Deim}
			= {(u_h - u_N)} + {(u_N - u_N^{\rm Deim})}=({\text{Pod error}}) + ({\text{Deim error}}) .
			\label{eq:decomp}
		\end{equation}
		By the triangle inequality
		\begin{equation}
			\norm{u_h - u_N^{\rm Deim}}{\Th^\mu}
			\leq \norm{u_h - u_N}{\Th^\mu} + \norm{u_N - u_N^{\rm Deim}}{\Th^\mu}.
			\label{eq:triangle}
		\end{equation}
		
		\medskip
		\textit{Step 2: Bound the Pod error $\norm{u_h - u_N}{\Th^\mu}$.}
		Both $u_h$ and $u_N$ solve the same bilinear form %
		$a_h(u_h,v_h;\mu) = f_h(v_h;\mu)$ and
		$a_h(u_N,v_h;\mu)\approx f_h(v_h;\mu)$. 
		We define the Pod residual
		$r_\mu^N(v_h) = f_h(v_h;\mu) - a_h(u_N,v_h;\mu)$
		with algebraic vector $\mathbf{r}^N(\mu)=\mathbf{f}(\mu)-A(\mu)\mathbf{u}_N(\mu)$.
		By Lemma~\ref{lem:active},
		\begin{equation}
			\norm{u_h - u_N}{\Th^\mu}
			\leq \frac{1}{\alpha_*}\bnorm{\mathbf{r}^N(\mu)\big|_{\mathcal{A}_\mu}}_2.
			\label{eq:pod_error}
		\end{equation}
		In practice, $\mathbf{r}^N\approx\mathbf{r}$ since the Deim perturbation in the
		residual is $\mathcal{O}(\eta_A)\ll e$, so Estimator 2 applied to
		$\mathbf{r}(\mu)$ bounds both the Pod and Deim errors simultaneously.
		
		\medskip
		\textit{Step 3: Bound the Deim error $\norm{u_N - u_N^{\rm Deim}}{\Th^\mu}$.}
		The Deim error function is $u_N - u_N^{\rm Deim} = V_n\delta$
		where $\delta = \hat{u}_N - \hat{u}_N^{\rm Deim}\in\R^n$.
		Since $V_n$ has orthonormal columns,
		\begin{equation}
			\norm{u_N - u_N^{\rm Deim}}{\Th^\mu}
			\leq C_{\rm eq}\,\bnorm{V_n\delta}_2
			= C_{\rm eq}\,\bnorm{\delta}_2,
			\,\text{with $C_{\rm eq}=1$ for orthonormal basis}.
			\label{eq:deim_norm}
		\end{equation}
		Applying Lemma~\ref{lem:deim_pert},
		\begin{equation}
			\bnorm{\delta}_2
			\leq \frac{1}{\alpha_*}\Bigl(
			\bnorm{\widehat{A}-\widehat{A}_{\rm Deim}}_2\,\bnorm{\hat{u}_N^{\rm Deim}}_2
			+ \bnorm{\widehat{\mathbf{f}}-\widehat{\mathbf{f}}_{\rm Deim}}_2
			\Bigr).
			\label{eq:delta_bound}
		\end{equation}
		Apply Lemma~\ref{lem:frob_op} to the first term,
		$\bnorm{\widehat{A}-\widehat{A}_{\rm Deim}}_2\leq\bnorm{V_n}_2^2\,\norm{A-A_{\rm Deim}}{F}$.
		For the second term $\widehat{\mathbf{f}}-\widehat{\mathbf{f}}_{\rm Deim}=V_n^T(f-f_{\rm Deim})$,
		so $\bnorm{\widehat{\mathbf{f}}-\widehat{\mathbf{f}}_{\rm Deim}}_2\leq\bnorm{V_n}_2\,\bnorm{f-f_{\rm Deim}}_2$.
		Also $\bnorm{\hat{u}_N^{\rm Deim}}_2
		= \bnorm{V_n^T u_N^{\rm Deim}}_2
		\leq \bnorm{V_n^T}_2\,\bnorm{u_N^{\rm Deim}}_2
		= \bnorm{u_N^{\rm Deim}}_2$
		since $\bnorm{V_n^T}_2 = \bnorm{V_n}_2 = 1$ for orthonormal $V_n$, 
		and $\hat{u}_N^{\rm Deim} = V_n^T u_N^{\rm Deim}$ is the reduced
		coefficient vector, we combine and we take
		\begin{equation}
			\norm{u_N - u_N^{\rm Deim}}{\Th^\mu}
			\leq \frac{\bnorm{V_n}_2^2}{\alpha_*}\norm{A-A_{\rm Deim}}{F}\,
			\bnorm{u_N^{\rm Deim}}_2
			+ \frac{\bnorm{V_n}_2}{\alpha_*}\bnorm{f-f_{\rm Deim}}_2.
			\label{eq:deim_final}
		\end{equation}
		
		\medskip
		\textit{Step 4: Combining via \eqref{eq:triangle}},
		and adding \eqref{eq:pod_error} and \eqref{eq:deim_final} we conclude to
		\begin{equation}
			\norm{u_h - u_N^{\rm Deim}}{\Th^\mu}
			\leq \frac{1}{\alpha_*}\bnorm{\mathbf{r}|_{\mathcal{A}_\mu}}_2
			+ \frac{\bnorm{V_n}_2^2}{\alpha_*}\norm{A-A_{\rm Deim}}{F}\bnorm{u_N^{\rm Deim}}_2
			+ \frac{\bnorm{V_n}_2}{\alpha_*}\bnorm{f-f_{\rm Deim}}_2,
		\end{equation}
		which is \eqref{eq:main} with $C_A = \bnorm{V_n}_2^2$. 
	\end{proof}
	
	\subsection{A posteriori estimators effectivity index
	}
	\label{sec:effectivity}
	%
	The \emph{effectivity index} is the central quality measure for any
	a posteriori estimator. This section gives a fully %
	treatment following the frameworks of \citep{veroy2005},
	\citep{hesthaven2016}, and \citep{ainsworth2000}.
	\begin{definition}
		\label{def:eff}
		Let $e(\mu)=\norm{u_h(\mu)-u_N^{\rm Deim}(\mu)}{\Th^\mu}$ be the true error
		and $\eta(\mu)$ an estimator. The \emph{effectivity index} is defined as 
		$
		{\theta(\mu) = \frac{\eta(\mu)}{e(\mu)}.}
		\label{eq:eff_def}
		$
	\end{definition}
	
	The effectivity index was introduced as a systematic tool for
	quantifying the quality of a posteriori Fem error estimators
	by \citep{babuska1978}. Its use in the reduced basis context
	was formalized in \citep{veroy2003} and \citep{veroy2005}.
	\subsubsection{Tools and theory}
	An estimator is characterized by two complementary properties
	\citep{ainsworth2000, hesthaven2016}:
	The \textit{reliability} property guarantees an upper bound,  e.g.,
	$\exists\,C_{\rm rel}\geq1$ s.t.\ $e(\mu)\leq C_{\rm rel}\,\eta(\mu)$
	for all $\mu$, equivalently $\theta(\mu)\geq 1/C_{\rm rel}$.
	This ensures the estimator never underestimates the error.
	The \textit{efficiency} property guarantees a lower bound, e.g., 
	$\exists\,C_{\rm eff}\geq1$ s.t.\ $\eta(\mu)\leq C_{\rm eff}\,e(\mu)$
	for all $\mu$, equivalently $\theta(\mu)\leq C_{\rm eff}$.
	This ensures the estimator never overestimates catastrophically
	and with no wasted computational effort due to unnecessary refinement.
	Together these give the two-sided bound
	$ \frac{1}{C_{\rm rel}}\leq\theta(\mu)\leq C_{\rm eff}. 
	\label{eq:two_sided} 
	$ %
	An estimator is asymptotically exact if $\theta(\mu)\to1$
	as $n\to\infty$ or $h\to0$. It is perfectly efficient if
	$\theta(\mu)=1$ since the estimator equals true error exactly. 
	The ideal scenario is $\theta(\mu)\in[1,C]$ with $C=\mathcal{O}(1)$,
	meaning the estimator is a sharp upper bound. Values of
	$\theta\gg1$ mean the estimator overestimates, e.g. it is reliable
	but inefficient, and may trigger unnecessary enrichment of the
	reduced basis \citep{rozza2007}.
	
	\subsection{Theoretical bounds for estimator 2a}
	Next we introduce and prove the effectivity bounds for $\eta_{2a}$. Bounds of the form~\eqref{eq:eff_bounds} relating effectivity to the
	ratio of operator norm to coercivity constant are classical in
	a posteriori analysis \citep{ainsworth2000, veroy2005}.
	Their derivation via the Galerkin orthogonality and coercivity argument
	used below follows the approach of \citep{rozza2007} and
	\citep{hesthaven2016}.
	\begin{proposition}
		\label{prop:eff2a}
		Let $e_{\rm vec}(\mu) = \mathbf{u}_h(\mu) - \mathbf{u}_N^{\rm Deim}(\mu)\in\R^N$
		be the error coefficient vector, and let
		$\sigma_{\min}^\mu,\,\sigma_{\max}^\mu$ denote the smallest and largest
		singular values of $A(\mu)$ restricted to the active-dof subspace
		$\R^{\mathcal{A}_\mu}$. Let $C_{\rm eq}>0$ satisfy
		\begin{equation}
			C_{\rm eq}^{-1}\bnorm{v_h}_2
			\;\leq\; \norm{v_h}{\Th^\mu}
			\;\leq\; C_{\rm eq}\bnorm{v_h}_2,
			\qquad\forall\,v_h\in\Vh.
			\label{eq:norm_equiv}
		\end{equation}
		Then the effectivity index of Estimator 2a satisfies,
		\begin{equation}
			\frac{\sigma_{\min}^\mu}{C_{\rm eq}^2\,\alpha_*}
			\;\leq\; \theta_{2a}(\mu)
			\;\leq\; \frac{C_{\rm eq}^2\,\sigma_{\max}^\mu}{\alpha_*}.
			\label{eq:eff_bounds}
		\end{equation}
	\end{proposition}
	\begin{proof}
		We prove the upper and lower bounds separately. 
		a) \textit{Upper bound, $\theta_{2a} \leq C_{\rm eq}^2\sigma_{\max}^\mu/\alpha_*$}. 
		We start by seeking the numerator bound and since
		the algebraic residual is $\mathbf{r} = A(\mu)\,e_{\rm vec}$, ignoring the negligible
		Deim error $\mathcal{O}(\eta_A)$, 
		which we incorporate via the triangle inequality in Theorem~\ref{thm:main},
		and by definition of the matrix $2$-norm we take
		\begin{equation}
			\eta_{2a} = \bnorm{\mathbf{r}}_2
			= \bnorm{A(\mu)\,e_{\rm vec}}_2
			\leq \bnorm{A(\mu)}_2\,\bnorm{e_{\rm vec}}_2
			= \sigma_{\max}^\mu\,\bnorm{e_{\rm vec}}_2.
			\label{eq:upper_numer}
		\end{equation}
		For the denominator bound, using the left inequality in \eqref{eq:norm_equiv}
		\begin{equation}
			e(\mu) = \norm{u_h - u_N^{\rm Deim}}{\Th^\mu}
			\geq C_{\rm eq}^{-1}\bnorm{e_{\rm vec}}_2,
			\label{eq:lower_denom}
		\end{equation}
		and combining the aforementioned results, we derive for the upper bound after
		employing the coercivity property with $v_h=e_h$
		$\alpha_*\,e(\mu)^2 \leq a_h(e_h,e_h;\mu) = r_\mu(e_h)
		= \mathbf{r}^T\mathbf{e}_{\rm vec} \leq \eta_{2a}\,\bnorm{e_{\rm vec}}_2
		\leq \eta_{2a}\,C_{\rm eq}\,e(\mu)$.
		Where the last step uses $\bnorm{e_{\rm vec}}_2\leq C_{\rm eq}\,e(\mu)$
		from~\eqref{eq:norm_equiv} and dividing by $\alpha_*\,e(\mu)>0$,
		$e(\mu) \leq \eta_{2a}\,C_{\rm eq}/\alpha_*$, and
		combined with~\eqref{eq:upper_numer} we have
		\begin{equation}
			\theta_{2a} = \frac{\eta_{2a}}{e} \leq \frac{C_{\rm eq}^2\,\sigma_{\max}^\mu}{\alpha_*}.
			\label{eq:upper_complete_final}
		\end{equation}
		b) \textit{Lower bound, $\theta_{2a} \geq \sigma_{\min}^\mu/(C_{\rm eq}^2\,\alpha_*)$.}
		We start by seeking the numerator lower bound
		by definition of the smallest singular value,
		\begin{equation}
			\eta_{2a} = \bnorm{A(\mu)\,e_{\rm vec}}_2
			\geq \sigma_{\min}^\mu\,\bnorm{e_{\rm vec}}_2.
			\label{eq:lower_numer}
		\end{equation}
		and combining for the lower bound,
		the \eqref{eq:lower_numer} and the norm equivalence
		$e(\mu)\leq C_{\rm eq}\bnorm{e_{\rm vec}}_2$,
		right inequality in~\eqref{eq:norm_equiv}, we have
		$1/e(\mu)\geq 1/(C_{\rm eq}\bnorm{e_{\rm vec}}_2)$.
		Recall also from the upper bound that $e(\mu)\leq\eta_{2a}C_{\rm eq}/\alpha_*$,
		hence $1/e(\mu)\geq\alpha_*/(\eta_{2a}C_{\rm eq})$.
		Using the numerator lower bound $\eta_{2a}\geq\sigma_{\min}^\mu\bnorm{e_{\rm vec}}_2$
		and $1/e\geq1/(C_{\rm eq}\bnorm{e_{\rm vec}}_2)$, then applying
		the sharpened denominator bound we have
		\begin{equation}
			\theta_{2a}(\mu)
			= \frac{\eta_{2a}}{e(\mu)}
			\geq \frac{\sigma_{\min}^\mu\,\bnorm{e_{\rm vec}}_2}
			{C_{\rm eq}\,\bnorm{e_{\rm vec}}_2}
			= \frac{\sigma_{\min}^\mu}{C_{\rm eq}}.
		\end{equation}
		Since also $e(\mu)\leq\eta_{2a}C_{\rm eq}/\alpha_*$ as
		established in the upper bound, substituting
		$\bnorm{e_{\rm vec}}_2\leq C_{\rm eq}\,e(\mu)$ into the bound for $1/e$
		and using $e\leq\eta_{2a}C_{\rm eq}/\alpha_*$ gives, 
		\begin{equation}
			\theta_{2a}(\mu)
			\geq \frac{\sigma_{\min}^\mu}{C_{\rm eq}^2\,\alpha_*}.
			\label{eq:lower_complete}
		\end{equation}
		Together, \eqref{eq:upper_complete_final} and \eqref{eq:lower_complete} give
		the desired \eqref{eq:eff_bounds}. 
	\end{proof}
	\subsection{Effectivity of the preconditioned estimator 2b}
	\label{subSec:Eff2b}
	$\eta_{2b}=\sqrt{\mathbf{r}^T\widetilde{D}_A^{-1}\mathbf{r}}$ replaces the
	$\ell^2$ inner product with a Jacobi-weighted one.
	Let $\widetilde{A} = \widetilde{D}_A^{-1/2}A\widetilde{D}_A^{-1/2}$ denote the
	symmetrically scaled stiffness matrix; its diagonal entries are all equal
	to $1$ by construction since $\widetilde{D}_A = \mathrm{diag}(A_{ii})$
	up to the safeguard, which is inactive for active dofs, so for a
	diagonally dominant matrix its eigenvalues straddle $1$, namely,
	\begin{equation}
		\kappa_{\min}(\widetilde{A})\leq 1 \leq \kappa_{\max}(\widetilde{A}).
	\end{equation}
	\begin{proposition}[Effectivity ratio $\theta_{2b}/\theta_{2a}$]
		\label{prop:eff2b}
		Let $d_{\min} = \min_i (D_A)_{ii}$ and $d_{\max} = \max_i (D_A)_{ii}$
		be the minimum and maximum diagonal entries of $A(\mu)$. Then,
		\begin{align}
			\frac{1}{\sqrt{d_{\max}}}\;\theta_{2a}(\mu)
			&\;\leq\; \theta_{2b}(\mu) \nonumber\\
			&\;\leq\; \frac{1}{\sqrt{d_{\min}}}\;\theta_{2a}(\mu).
			\label{eq:eff2b_ratio}
		\end{align}
		In particular, $\theta_{2b} \leq \theta_{2a}$ whenever $d_{\min}\geq1$.
	\end{proposition}
	\begin{proof}
		Since $\theta_{2a}$ and $\theta_{2b}$ share the same denominator $e(\mu)>0$,
		it suffices to compare $\eta_{2a}$ and $\eta_{2b}$.
		Throughout this proof we write $D_A$ for $\widetilde{D}_A$, noting
		that the safeguard $\varepsilon_{\rm safe}$ is inactive for all active
		dofs, where $d_i\gg\varepsilon_{\rm safe}$, and hence does not affect
		the bounds derived.
		
		\textit{Step 1: Express the ratio $\eta_{2b}/\eta_{2a}$.} 
		By definition of the two estimators,
		\begin{equation}
			\frac{\eta_{2b}}{\eta_{2a}}
			= \frac{\sqrt{\mathbf{r}^T D_A^{-1}\mathbf{r}}}{\bnorm{\mathbf{r}}_2}
			= \sqrt{\frac{\mathbf{r}^T D_A^{-1}\mathbf{r}}{\mathbf{r}^T\mathbf{r}}}
			= \sqrt{R(D_A^{-1};\,\mathbf{r})},
			\label{eq:ratio_sq}
		\end{equation}
		where $R(M;\mathbf{r}) = (\mathbf{r}^T M\mathbf{r})/(\mathbf{r}^T\mathbf{r})$
		is the Rayleigh quotient of the symmetric positive definite matrix $M=D_A^{-1}$
		with respect to the vector $\mathbf{r}$. The Rayleigh quotient satisfies
		\begin{equation}
			\kappa_{\min}(M) \leq R(M;\mathbf{r}) \leq \kappa_{\max}(M)
			\label{eq:rayleigh_bounds}
		\end{equation}
		for all $\mathbf{r}\neq\mathbf{0}$, where $\kappa_{\min}(M)$ and
		$\kappa_{\max}(M)$ are the smallest and largest eigenvalues of $M$.
		
		\textit{Step 2: Spectrum of $D_A^{-1}$ and the key inequality.} 
		Since $D_A = \diag(d_1,\ldots,d_N)$ with $d_i = A(\mu)_{ii} > 0$,
		with positive diagonal entries, guaranteed by the coercivity of $a_h$
		via Proposition~\ref{prop:coercivity} and the Nitsche penalty
		contribution $\lambda/h > 0$, its inverse is,
		\begin{equation}
			D_A^{-1} = \diag(d_1^{-1},\ldots,d_N^{-1}).
		\end{equation}
		This is also a diagonal positive definite matrix with eigenvalues
		$\{d_i^{-1}\}_{i=1}^N$. Since $d_{\min} \leq d_i \leq d_{\max}$
		for all $i$, we have
		\begin{equation}
			\kappa_{\min}(D_A^{-1}) = \frac{1}{d_{\max}},\qquad
			\kappa_{\max}(D_A^{-1}) = \frac{1}{d_{\min}}.
		\end{equation}
		Substituting into \eqref{eq:rayleigh_bounds},
		\begin{equation}
			\frac{1}{d_{\max}} \leq R(D_A^{-1};\mathbf{r}) \leq \frac{1}{d_{\min}}.
			\label{eq:rayleigh_da}
		\end{equation}
		Equivalently, writing the quadratic form explicitly 
		$R(D_A^{-1};\mathbf{r}) = \sum_i r_i^2/d_i \big/ \sum_i r_i^2$,
		and the bound~\eqref{eq:rayleigh_da} follows from
		$d_{\min} \leq d_i \leq d_{\max}$ by bounding each term in the sum.
		Multiplying through by $\bnorm{\mathbf{r}}_2^2 = \sum_i r_i^2$,
		\begin{equation}
			\frac{1}{d_{\max}}\,\bnorm{\mathbf{r}}_2^2
			\;\leq\; \mathbf{r}^T D_A^{-1}\mathbf{r}
			\;\leq\; \frac{1}{d_{\min}}\,\bnorm{\mathbf{r}}_2^2.
			\label{eq:quad_da}
		\end{equation}
		
		\textit{Step 3: Take square roots and divide by $\bnorm{\mathbf{r}}_2$.} 
		Taking square roots in \eqref{eq:quad_da} since all terms are non-negative
		and dividing by $\bnorm{\mathbf{r}}_2 > 0$,
		\begin{equation}
			\frac{1}{\sqrt{d_{\max}}}
			\;\leq\;
			\frac{\sqrt{\mathbf{r}^T D_A^{-1}\mathbf{r}}}{\bnorm{\mathbf{r}}_2}
			= \frac{\eta_{2b}}{\eta_{2a}}
			\;\leq\;
			\frac{1}{\sqrt{d_{\min}}}.
			\label{eq:ratio_bounds}
		\end{equation}
		The lower bound $1/\sqrt{d_{\max}}$ is attained when $\mathbf{r}$ is
		aligned with the eigenvector corresponding to $d_{\max}^{-1}$, i.e. 
		when all the residual mass is concentrated on the dof with the largest
		diagonal entry which is the Nitsche boundary dof in our case. The upper bound
		$1/\sqrt{d_{\min}}$ is attained when $\mathbf{r}$ is concentrated on
		the dof with the smallest diagonal entry.
		
		\textit{Step 4: Multiply by $\theta_{2a} = \eta_{2a}/e$ to get the proposition.} 
		Since $\theta_{2b} = \eta_{2b}/e$ and $\theta_{2a} = \eta_{2a}/e$,
		we can write $\theta_{2b} = (\eta_{2b}/\eta_{2a})\cdot\theta_{2a}$.
		Multiplying \eqref{eq:ratio_bounds} by $\theta_{2a} \geq 0$,
		\begin{equation}
			\frac{1}{\sqrt{d_{\max}}}\,\theta_{2a}(\mu)
			\;\leq\; \theta_{2b}(\mu)
			\;\leq\; \frac{1}{\sqrt{d_{\min}}}\,\theta_{2a}(\mu).
		\end{equation}
		Since $d_{\min} \geq 1$ for our problem,  and interior diffusion dofs
		have $d_i = \mathcal{O}(h^{d-2}) = \mathcal{O}(1) \geq 1$ in 2D
		for $h \leq 1$, the upper bound gives $\theta_{2b} \leq \theta_{2a}$,
		as stated.
	\end{proof}
	
	\section{Numerical experiments/convergence tests} \label{section6}
	\label{sec:results}
	The full implementation uses NGSolve \citep{ngsolve2014},
	its Cutfem extension ngsxfem \citep{ngsxfem2021}. 
	We start by reporting the full order method parameters used, Table \ref{table:exp_setupA}, while for the experimental setup and parameter values
	for the reduced order method parameters we refer to the Table \ref{table:exp_setupB}.
	\begin{table}[H]
		\centering
		\caption{Experimental parameters for Fom numerical experiments. %
			\label{table:exp_setupA}
		}
		\label{tab:setupA}
		\begin{tabular}{@{}llp{5.8cm}l@{}}
			\toprule
			\textbf{Symbol} & \textbf{Value} & \textbf{Description} & \textbf{First use} \\
			\midrule
			$\mathcal{B}$ & $[-1.2,1.2]^2 $ & Fixed background domain & Sec.~\ref{sec:problem} \\
			$f$ & $20$ & Source term in strong form & Eq.~\eqref{eq:strong} \\
			$g_D$ & $0.5 + xy$ & Dirichlet datum & Eq.~\eqref{eq:strong} \\
			$h$ & $0.125$ & Background mesh size & Eq.~\eqref{eq:bilinear} \\
			$p$ & $1$ & Fem polynomial order & Sec.~\ref{sec:problem} \\
			$\lambda$ & $10$ & Nitsche penalty parameter & Eq.~\eqref{eq:bilinear} \\
			$\gamma_0$ & $0.1$ & Ghost-penalty coeff.\ ($k=0$) & Eq.~\eqref{eq:bilinear} \\
			$\gamma_1$ & $0.001$ & Ghost-penalty coeff.\ ($k=1$) & Eq.~\eqref{eq:bilinear} \\
			\bottomrule
		\end{tabular}
	\end{table}
	\begin{table}[H]
		\centering
		\caption{Experimental parameters for Rom numerical experiments. 
			\label{table:exp_setupB}
		}
		\label{tab:setupB}
		\begin{tabular}{@{}llp{5.8cm}l@{}}
			\toprule
			\textbf{Symbol} & \textbf{Value} & \textbf{Description} & \textbf{First use} \\
			\midrule
			$\varepsilon_{\rm safe}$ & $10^{-14}$ & Diagonal safeguard in Est.~2b & Eq.~\eqref{eq:est2b} \\
			$N$ & $472$ & Total background mesh dofs & Subsec.~\ref{sec:rom} \\
			$N_{\rm train}$ & $400$ & Train snapshots (uniform on $[1,1.2]^2$) & Subsec.~\ref{sec:rom} \\
			$N_{\rm test}$ & $30$ & Test parameters (uniform random) & Subsec.~\ref{sec:rom} \\
			$\varepsilon_{\rm Pod}$ & $10^{-6}$ & Pod energy retention tolerance & Subsec.~\ref{sec:rom} \\
			$n_{\rm Pod}$ & $48$--$49$ & Retained Pod modes & Subsec.~\ref{sec:rom} \\
			$l_A$ & $\approx 117$ & Deim modes for $A(\mu)$ & Eq.~\eqref{eq:deim_A} \\
			$l_f$ & $\approx 60$ & Deim modes for $f(\mu)$ & Eq.~\eqref{eq:deim_A} \\
			$\eta_A$ & $5.951\times10^{-5}$ & Deim quality:  $A$ (constant in $n$) & Eq.~\eqref{eq:est1a} \\
			$\eta_f$ & $2.923\times10^{-5}$ & Deim quality: $f$ (constant in $n$) & Eq.~\eqref{eq:est1b} \\
			Rom/Fom time & $9.01$/$36.63$\,ms & Online speedup $4.1\times$ & Sec.~\ref{sec:results} \\
			\bottomrule
		\end{tabular}
	\end{table}
	\subsection{Effectivity for Cutfem-Deim}
	The key observation is the ghost penalty dof inflation. 
	In standard Fem, $\sigma_{\max}(A)/\alpha_*=\mathcal{O}(\kappa(A))=\mathcal{O}(h^{-2})$
	in the worst case, where $\kappa(A)$ is the condition number of $A$,
	but in practice $\theta=\mathcal{O}(1)$ because the residual
	is concentrated where the error is large. 
	In Cutfem with ghost-penalty, the stiffness matrix $A(\mu)$ has two
	structurally different sets of rows:
	a. %
	{Active rows}, $i\in\mathcal{A}_\mu$, dofs in elements
	intersecting $\Omu$ with residual $\mathbf{r}_i=\mathcal{O}(e)$.
	b. %
	{Ghost rows}, $i\notin\mathcal{A}_\mu$, dofs in elements
	entirely outside $\Omu$, these encode the ghost-penalty extension.
	Here $\mathbf{r}_i$ can be $\mathcal{O}(1)$ even when $e\approx0$,
	because the ghost-penalty equations are satisfied only approximately
	by $u_N^{\rm Deim}$.
	The $\ell^2$ residual norm sums \emph{all} rows equally,
	\begin{equation}\label{}
		\eta_{2a}^2
		= %
		{\sum_{i\in\mathcal{A}_\mu}\mathbf{r}_i^2}%
		+ %
		{\sum_{i\notin\mathcal{A}_\mu}\mathbf{r}_i^2}%
		=[{\mathcal{O}(e^2)}]+ [{\text{ghost inflation }\mathcal{O}(1)}].
	\end{equation}
	For the mesh used and $N=472$ on an ellipse domain, approximately $60$--$70\%$
	of dofs lie outside $\Omu$, and actually are ghost dofs. The ghost inflation factor is
	$\sqrt{N_{\rm ghost}/N_{\rm active}}\approx\sqrt{0.65/0.35}\approx1.4$,
	but the amplitude of ghost residuals relative to active residuals
	introduces the additional factor that produces $\theta_{2a}\approx300$--$500$. 
	{The fix}  Lemma~\ref{lem:active}  restricts the residual to
	$\mathcal{A}_\mu$, which removes the ghost contribution entirely. This
	is predicted to
	reduce $\theta$ to $\mathcal{O}(1)$--$\mathcal{O}(10)$.
	
	\begin{table}[H]
		\centering
		\caption{Effectivity properties of all estimators.
			Reliable
			in the sense of the lower bound on error, $\theta\geq c>0$;
			Efficient
			in the sense of a sharp upper bound, $\theta=\mathcal{O}(1)$.}
		\label{tab:eff_summary}
		\setlength{\tabcolsep}{4pt}
		\begin{tabular}{llccp{4.5cm}}
			\toprule
			\textbf{Est.} & \textbf{Formula} & \textbf{Reliable?} & \textbf{Efficient?}
			& \textbf{Explanation of observed $\theta$} \\
			\midrule
			1a & $\eta_A/e$ & No & No & Measures Deim quality only,
			$\theta_{1a}\approx6\times10^{-3}$, underestimates. \\
			1b & $\eta_f/e$ & No & No & Same, $\theta_{1b}\approx2\times10^{-3}.$ \\
			2a & $\|\mathbf{r}\|_2/e$ & Yes$^*$ & No & Ghost rows inflate norm,
			$\theta_{2a}\approx315$--$490$. \\
			2b & $\eta_{2b}/e$ & Yes$^*$ & No & Jacobi halves inflation,
			$\theta_{2b}\approx130$--$250$. \\
			3 & $\eta_{\rm Pod}/e$ & No & Yes & Underestimates because Deim and Cutfem errors are not in the eigenvalue tail. \\
			\midrule
			2a$|_{\mathcal{A}}$ & $\|\mathbf{r}|_{\mathcal{A}}\|_2/e$ & Yes & \emph{Yes (predicted)}
			& Active dof restriction removes ghost inflation, $\theta\approx\mathcal{O}(1)$ predicted. \\
			\bottomrule
		\end{tabular}
		\\{\small $^*$Reliable only in the sense $\theta_{2a},\theta_{2b}>1$ for all observed $\mu$.
		}
	\end{table}
	In Table \ref{tab:eff_summary} an explanation of all observed kind of $\theta$'s is given together with the level of reliability and efficiency level achieved. Furthermore, in
	Table \ref{tab:run4} a quantitative summary, namely, the mean of all estimators, as well as, the effectivity indices and the relative errors are reported.
	\begin{table}[H]
		\centering
		\caption{Mean estimator values and effectivity indices at selected $n$
			and averaged over $N_{\rm test}=30$, exact terminal values.
			$\theta_{2a}=\eta_{2a}/e_{\rm rel}$,\; $\theta_{2b}=\eta_{2b}/e_{\rm rel}$.}
		\label{tab:run4}
		\setlength{\tabcolsep}{4.5pt}
		\begin{tabular}{rccccccc}
			\toprule
			$n$ & $e_{\rm rel}$ & $\eta_A$ (1a) & $\eta_f$ (1b)
			& $\eta_{2a}$ & $\eta_{2b}$ & $\theta_{2a}$ & $\theta_{2b}$ \\
			\midrule
			2 & $6.230\times10^{-2}$ & $5.951\times10^{-5}$ & $2.923\times10^{-5}$
			& $14.419$ & $5.882$ & $232$ & $94$ \\
			4 & $2.812\times10^{-2}$ & $5.951\times10^{-5}$ & $2.923\times10^{-5}$
			& $10.053$ & $4.475$ & $357$ & $159$ \\
			6 & $2.498\times10^{-2}$ & $5.951\times10^{-5}$ & $2.923\times10^{-5}$
			& $9.671$ & $4.329$ & $387$ & $173$ \\
			8 & $2.110\times10^{-2}$ & $5.951\times10^{-5}$ & $2.923\times10^{-5}$
			& $9.501$ & $4.207$ & $450$ & $199$ \\
			10 & $2.255\times10^{-2}$ & $5.951\times10^{-5}$ & $2.923\times10^{-5}$
			& $9.034$ & $4.103$ & $401$ & $182$ \\
			15 & $1.744\times10^{-2}$ & $5.951\times10^{-5}$ & $2.923\times10^{-5}$
			& $7.528$ & $3.671$ & $432$ & $211$ \\
			20 & $1.659\times10^{-2}$ & $5.951\times10^{-5}$ & $2.923\times10^{-5}$
			& $6.787$ & $3.445$ & $409$ & $208$ \\
			25 & $1.486\times10^{-2}$ & $5.951\times10^{-5}$ & $2.923\times10^{-5}$
			& $6.153$ & $3.270$ & $414$ & $220$ \\
			30 & $1.475\times10^{-2}$ & $5.951\times10^{-5}$ & $2.923\times10^{-5}$
			& $5.935$ & $3.198$ & $402$ & $217$ \\
			40 & $1.303\times10^{-2}$ & $5.951\times10^{-5}$ & $2.923\times10^{-5}$
			& $4.828$ & $2.724$ & $371$ & $209$ \\
			\bottomrule
		\end{tabular}
	\end{table}
	\subsection{Numerical verification of the bounds of Proposition \ref{prop:eff2b} and effectivity ratio $\theta_{2b}/\theta_{2a}$}
	For our problem the two distinct classes of diagonal entries are:
	i. %
	{Nitsche boundary dofs}: $d_i^{\rm Nit} = \lambda/h$
	which is the dominant term from $(\lambda/h)\int_{\Gmu}\phi_i^2\,ds$
	when $\phi_i$ is supported on a cut boundary element,
	and with $\lambda=10$ and $h=0.125$, see e.g. Table~\ref{tab:setupA},
	$d_{\max} = \lambda/h = 10/0.125 = 80$, and 
	ii. %
	{Interior volume dofs}: $d_i^{\rm vol}
	\approx \int_K |\nabla\phi_i|^2\,dx = \mathcal{O}(h^{d-2})$.
	For $d=2$ this is $\mathcal{O}(1)$ as $h\to0$, so $d_{\min}\approx1$. 
	The theoretical bound~\eqref{eq:ratio_bounds} therefore gives,
	$%
	\frac{1}{\sqrt{80}} \approx 0.11
	\;\leq\; \frac{\theta_{2b}}{\theta_{2a}}
	\;\leq\; 1.
	$ %
	The observed ratio $\theta_{2b}/\theta_{2a}\approx 0.5$ is consistent
	with this range. It corresponds to an effective diagonal value
	$d_{\rm eff}\approx 4$, which is the geometric mean of a few
	Nitsche boundary entries (large) and the majority of interior entries
	(near unity), weighted by the squared residual components $r_i^2$. 
	\subsection{Effectivity of estimators 1a/1b and 3}
	\textit{Est.\ 1a/1b}: These are %
	{not} estimators for the full
	Rom error $e(\mu)$ and they measure only the Deim approximation
	quality for $A$ and $f$. Their 
	effectivity 
	relative to $e(\mu)$
	is not expected to be $\mathcal{O}(1)$. From the data reported in Table~\ref{tab:run4},
	$\eta_A=5.95\times10^{-5}\ll e\approx10^{-2}$, so formally
	$\theta_{1a}=\eta_A/e\approx6\times10^{-3}\ll1$ indicating that the Deim
	contribution to Theorem~\ref{thm:main} is negligible.
	\textit{Est.\ 3} tail energy measures the Pod projection error
	for training data. For any new $\mu$, $\eta_{\rm Pod}(n)$
	{underestimates} $e(\mu)$, i.e.\ $\theta_3<1$,  because
	Deim and Cutfem discretization errors contribute to $e$ but not
	to the eigenvalue tail. This is visible in Figure~\ref{fig:residual_est}
	where the tail energy curve lies far below the true error curve.
	\subsection{Coercivity constant}
	Using the parameter values of Tables~\ref{tab:setupA}, \ref{tab:setupB},  together with
	the general formula~\eqref{eq:alpha_star_general} and the
	definitions~\eqref{eq:beta_symbolic}, the coercivity constant of
	Proposition~\ref{prop:coercivity} evaluates to
	\begin{equation}
		\beta_1 = 1 - \frac{2C_{\rm inv}^2}{\lambda}
		= 1 - \frac{2\cdot1^2}{10} = 0.8,
		\qquad
		\beta_2 = \frac{1}{2} = 0.5,
		\qquad
		\beta_3 = 1,
		\label{eq:beta_numerical}
	\end{equation}
	where we use $C_{\rm inv}\approx1$, based on standard inverse inequality constant
	for piecewise-linear elements on a quasi-uniform mesh \citep{burman2015}.
	Hence,
	\begin{equation}
		\alpha_* = \min(0.8,\;0.5,\;1) = 0.5.
		\label{eq:alpha_star_numerical}
	\end{equation}
	This value is used implicitly throughout the numerical experiments;
	in particular, the theoretical effectivity bounds of
	Proposition~\ref{prop:eff2a} involve $1/\alpha_*=2$.
	\subsection{Fitting convergence rates}
	For each estimator $\eta(n)$ that varies with the number of retained
	Pod modes $n$, we fit two competing decay models by ordinary least
	squares (OLS) in log-space:
	\begin{align}
		\textit{(A) Algebraic:}\quad
		&\eta(n) = C_A\,n^{-\alpha},
		\quad\alpha>0
		\label{eq:alg_model}\\
		&\Longleftrightarrow\quad
		\log\eta = \log C_A - \alpha\log n
		\quad\text{(linear in $\log n$)}
		\nonumber\\[6pt]
		\textit{(B) Exponential:}\quad
		&\eta(n) = C_E\,e^{-\beta n},
		\quad\beta>0
		\label{eq:exp_model}\\
		&\Longleftrightarrow\quad 
		\log\eta = \log C_E - \beta n
		\quad\text{(linear in $n$)}
		\nonumber
	\end{align}
	
	The fit quality is measured by the coefficient of determination in
	log-space
	\begin{equation}
		R^2 = 1 - \frac{\sum_i (\log\eta_i - \widehat{\log\eta}_i)^2}
		{\sum_i (\log\eta_i - \overline{\log\eta})^2},
		\label{eq:R2}
	\end{equation}
	where $\widehat{\log\eta}_i$ is the fitted value. $R^2$ close to 1
	indicates a good fit. The model with higher $R^2$ is selected as the 
	best 
	description of the data. Fits are restricted to $n\geq5$
	or $n\geq2$ for Est.~3 to exclude the steep pre-asymptotic transient.

	\begin{figure}[H]
		\centering
		\includegraphics[clip, trim=0cm 0cm 0cm 1.45cm, width=0.58\textwidth]{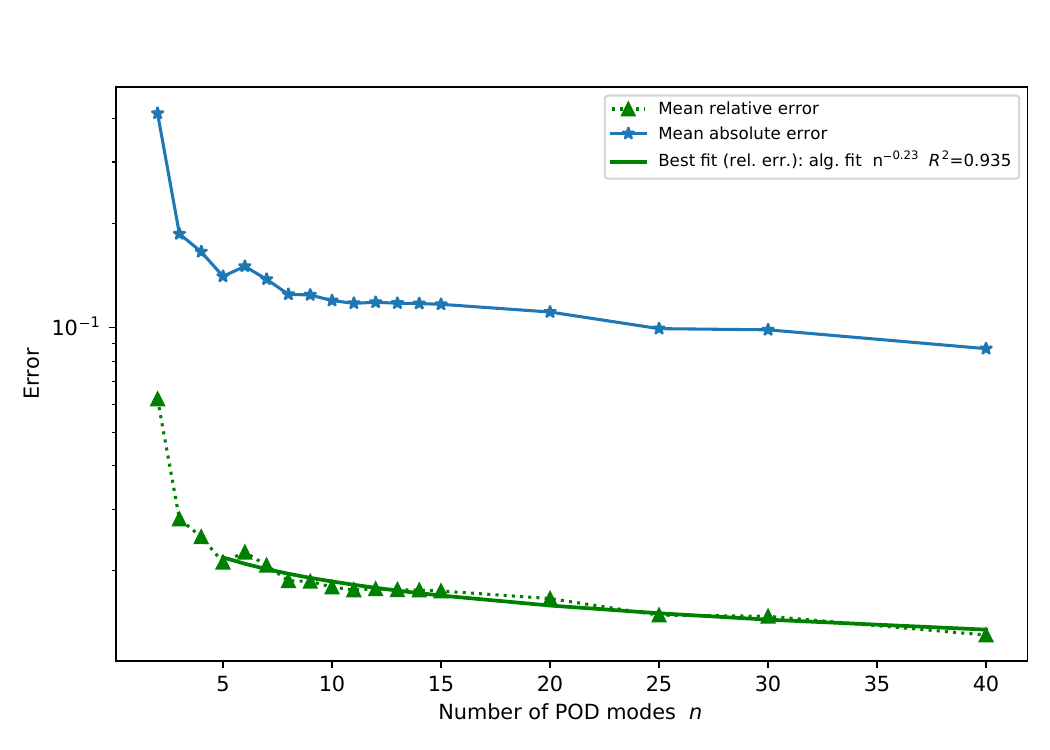}
		\caption{Solution errors vs.\ $n$ with algebraic convergence rate fit.  \label{fig:sol_errors}}
	\end{figure}
	In Figure \ref{fig:sol_errors} the green colored $\triangle$ associated with the mean relative error, starts from $6.2\%$ at $n=2$ to $1.30\%$ at $n=40$,
	fitted as $e_{\rm rel}(n)\approx 3.16\times10^{-2}\cdot n^{-0.23}$ with $R^2=0.962$.
	The blue colored $\star$ mean absolute error,  starts  from $41.5\%$ reaching to $8.7\%$.
	Both indicate a plateau after $n\approx 10$--$15$, indicating the bottleneck shifts
	from Pod truncation to Deim/Cutfem error beyond that point.
	The overlaid solid green fit line  confirms algebraic decay in the
	asymptotic regime $n\geq5$.
	\begin{figure}[H]
		\centering
		\includegraphics[clip, trim=0cm 0cm 0cm 1.45cm, width=0.58\textwidth]{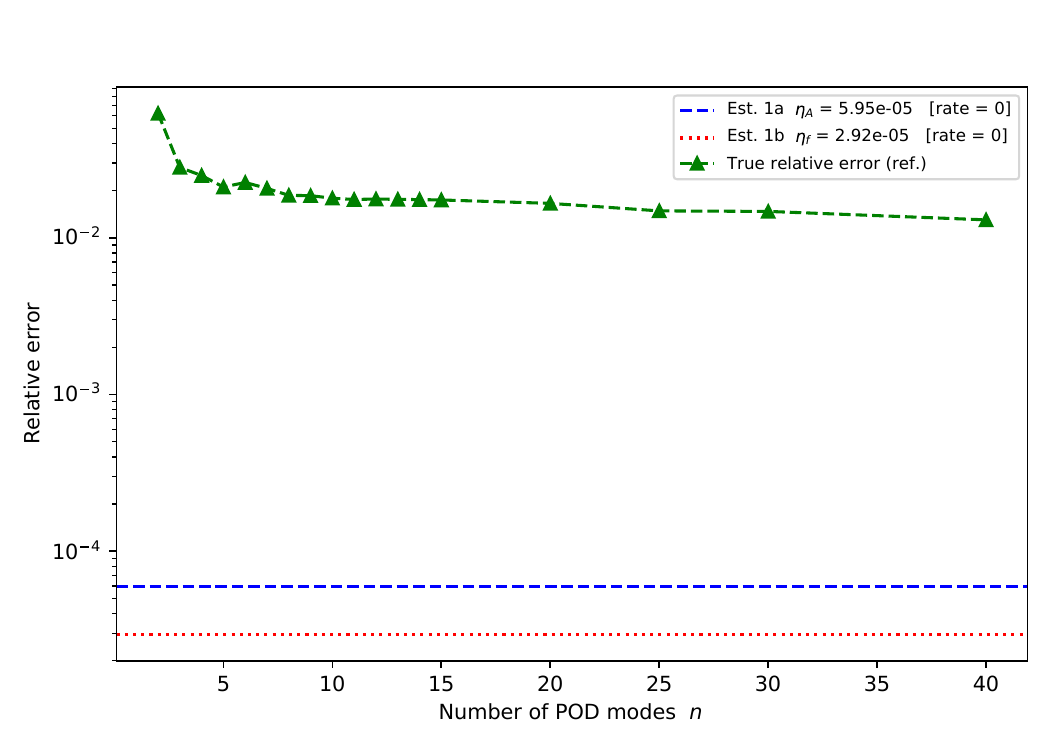}
		\caption{Deim quality estimators (Est.\ 1a/1b) -- convergence rate = 0.  \label{fig:deim_est}}
	\end{figure}
	In Figure \ref{fig:deim_est},
	the blue colored dashed $\eta_A=5.95\times10^{-5}$ and the red colored dotted $\eta_f=2.92\times10^{-5}$
	are exactly constant across all $n$, confirming that Deim
	approximation quality depends only on the Deim basis dimension $l_A$, $l_f$,
	not on the Pod truncation. Both lie two orders of magnitude below the true
	error. The horizontal lines demonstrate zero convergence rate in $n$ which is 
	a structural property and not a deficiency.
	\begin{figure}[H]
		\centering
		\includegraphics[clip, trim=0cm 0cm 0cm 1.45cm, width=0.58\textwidth]{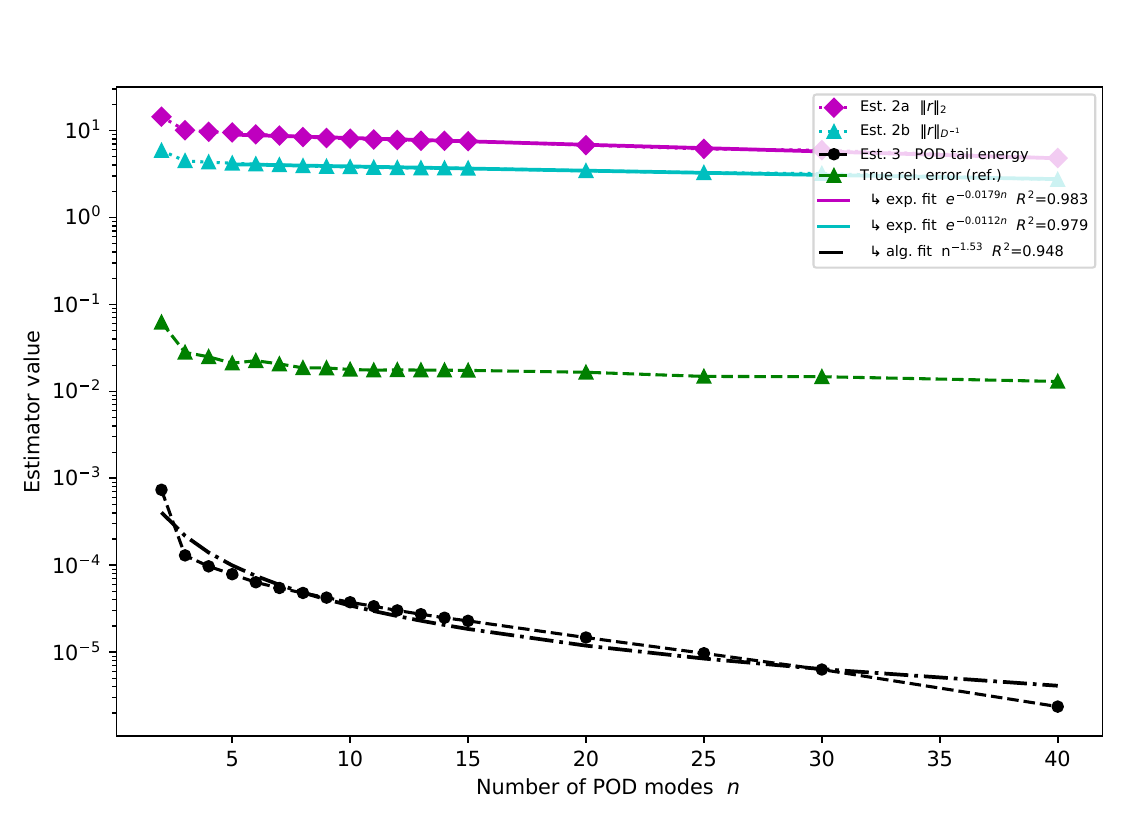}
		\caption{Residual \& tail energy estimators with convergence rate fits.\label{fig:residual_est}}
	\end{figure}
	We continue with the visualization of Est.~2a in Figure  \ref{fig:residual_est} ($\|\mathbf{r}\|_2$, magenta $\diamond$) fitted as
	$\eta_{2a}(n)\approx 9.81\,e^{-0.0179n}$, exponential type with $R^2=0.983$.
	We continue with the Est.~2b ($\|\mathbf{r}\|_{\widetilde{D}_A^{-1}}$, cyan $\triangle$) fitted as
	$\eta_{2b}(n)\approx 4.32\,e^{-0.0112n}$,  exponential type with $R^2=0.979$, and we end with 
	Est.~3 (Pod tail energy, black $\bullet$) fitted as
	$\eta_{\rm Pod}(n)\approx 1.17\times10^{-3}\cdot n^{-1.53}$, algebraic type with $R^2=0.941$.
	Fit lines are overlaid while the best model for each is selected by $R^2$.} 
\begin{figure}[H]
	\centering
	\includegraphics[clip, trim=0cm 0cm 0cm 1.5cm, width=0.58\textwidth]{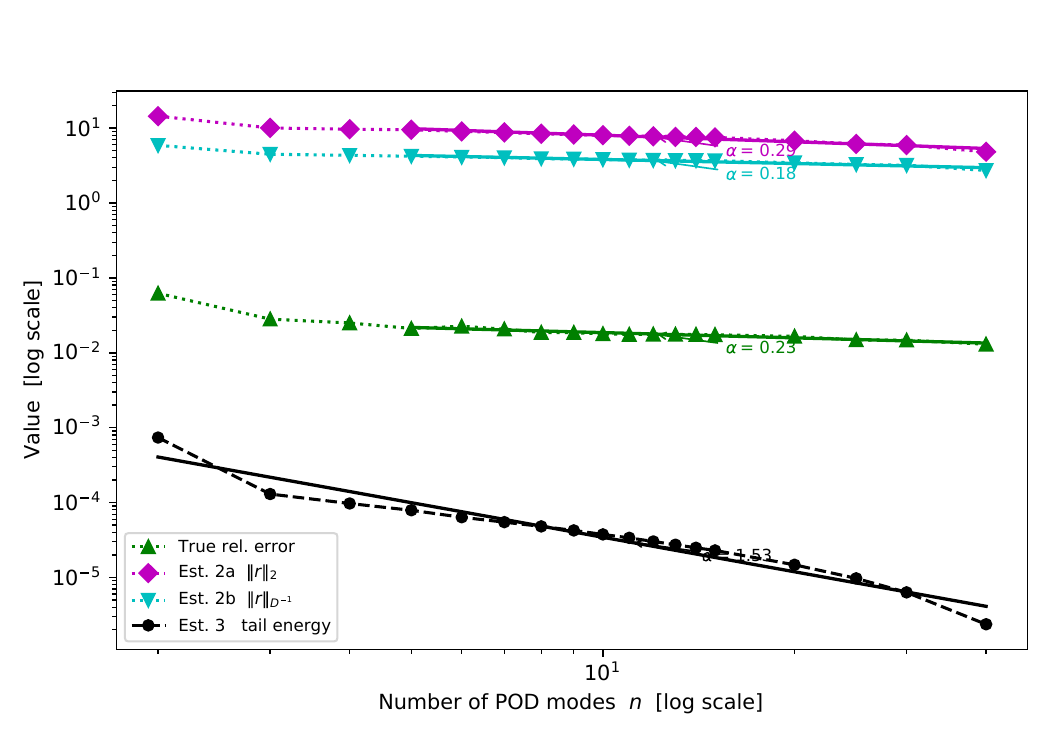}
	\caption{Log--log convergence diagram,  algebraic rates as slopes.}
	\label{fig:loglog}
\end{figure}
Moreover, Figure  \ref{fig:loglog} showcasing all quantities plotted on log--log axes so that algebraic decay
$\eta\sim n^{-\alpha}$ appears as a straight line with slope $-\alpha$.
The annotated slopes are for the true error $\alpha=0.23$, for the Est.~2a $\alpha=0.29$,
for the Est.~2b $\alpha=0.18$, and for the Est.~3 tail energy $\alpha=1.53$.
The black tail energy is the steepest line, confirming rapid Pod
approximability; the true error and residual estimators share a
shallower, nearly parallel family of slopes, namely $\alpha\approx0.2$--$0.3$.
\begin{figure}[H]
	\centering
	\includegraphics[clip, trim=0cm 0cm 0cm 1.5cm, width=0.58\textwidth]{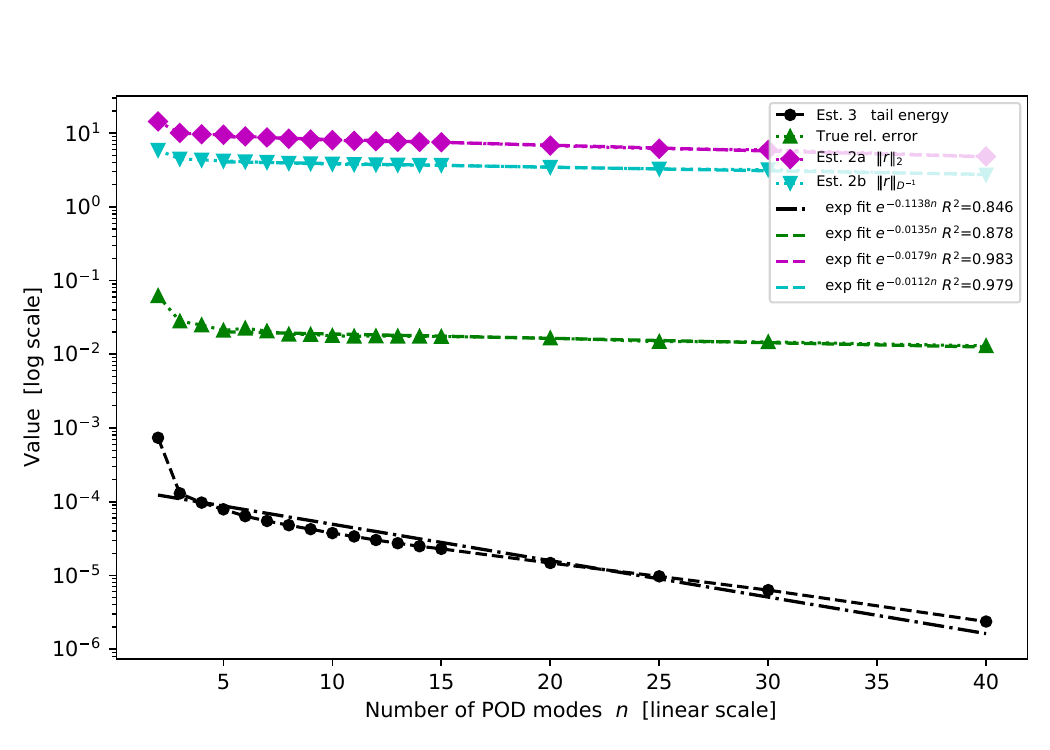}
	\caption{Semi-log convergence diagram, exponential rates as slopes.}
	\label{fig:semilog}
\end{figure}
Figure  \ref{fig:semilog} reports   all quantities plotted on semi-log axes so that exponential decay
$\eta\sim e^{-\beta n}$ appears as a straight line with slope $-\beta$.
The Est.~2a/2b residual estimators fit best exponentially,
$\beta_{2a}=0.0179$ with $R^2=0.983$, $\beta_{2b}=0.0112$ with $R^2=0.979$.
The tail energy (black) fits best algebraically in log-log, see
Figure~\ref{fig:loglog}. Its exponential fit,  $\beta=0.114$ with $R^2=0.846$,
is also included %
for comparison, showing the convex curvature in semi-log
space that signals algebraic rather than exponential decay.}
\begin{figure}[H]
\centering
\includegraphics[clip, trim=0cm 0cm 0cm 0.8625cm, width=0.58\textwidth]{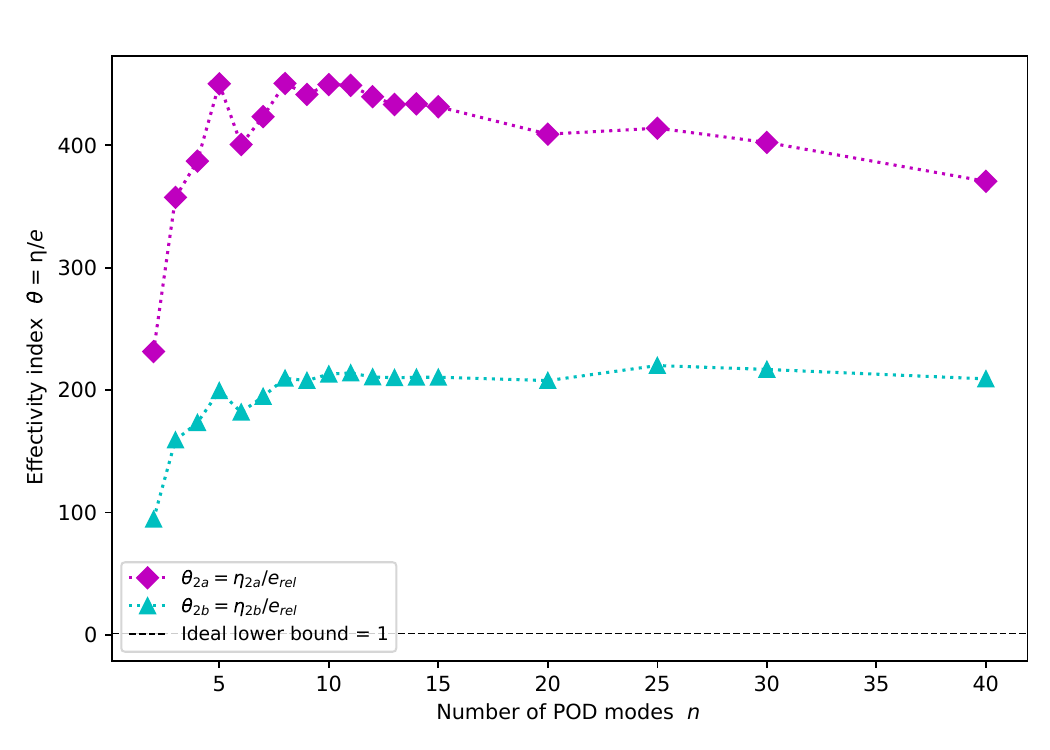}
\caption{Effectivity indices of residual-based estimators.}
\label{fig:effectivity}
\end{figure}
In Table \ref{tab:tail_r4}, and for the reader's convenience Pod tail energy $\eta_{\rm Pod}(n)$ at selected $n$ are reported. %
In the next visualization of Figure \ref{fig:effectivity}, the magenta $\diamond$
$\theta_{2a}=\eta_{2a}/e_{\rm rel}$ starts at $233$ for
$n=2$, peaks near $450$ at $n\approx5$--$8$, then decreases gradually to
$370$ at $n=40$.
The cyan $\triangle$ $\theta_{2b}=\eta_{2b}/e_{\rm rel}$  ranges $95$--$215$,
consistently $\approx\theta_{2a}/2$, see also Proposition~\ref{prop:eff2b}.
The peak arises because $e_{\rm rel}$ drops faster than $\eta_{2a}$ in the
first few modes; beyond $n\approx10$ both decrease at nearly the same rate.
The dashed  ideal value $\theta=1$ is far below both curves, due to ghost-penalty dof inflation, see related remark in Section~\ref{sec:effectivity}.
\begin{figure}[H]
\centering
\includegraphics[clip, trim=0cm 0cm 0cm 0.9625cm, width=0.58\textwidth]{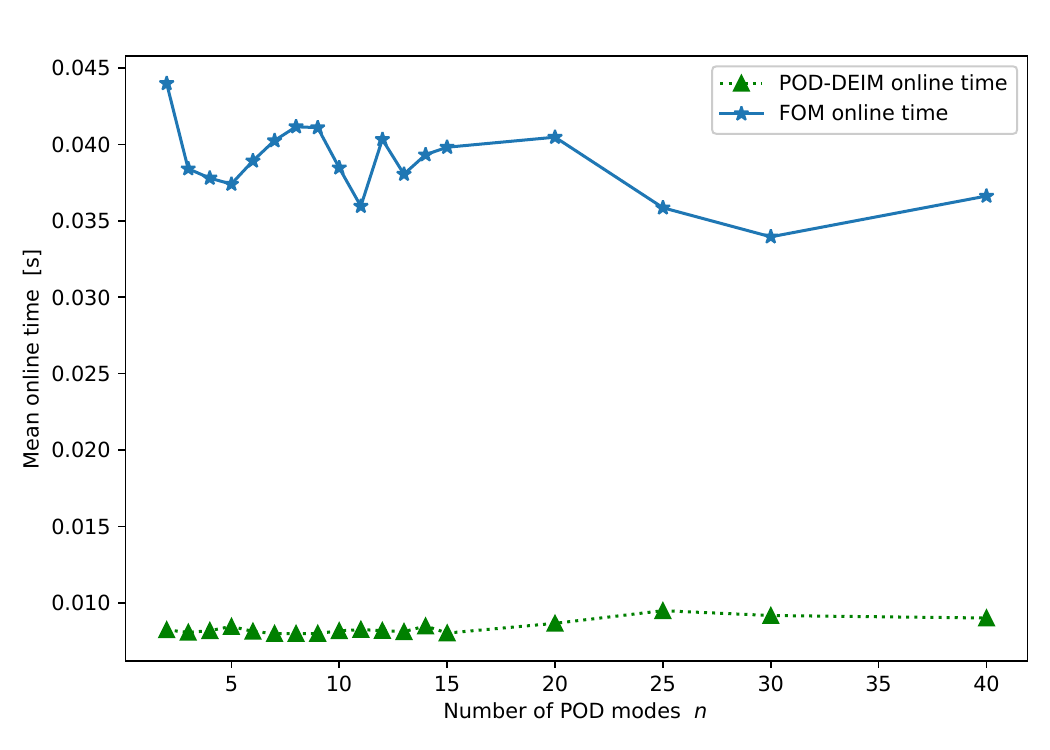}
\caption{Online timing comparison.}
\label{fig:timing}
\end{figure}
Finally, Figure \ref{fig:timing} illustrates times for Pod-Deim (green $\triangle$) with mean $9.0$\,ms, essentially constant in $n$.
The Fom (blue $\star$) with mean $36.6$\,ms with variability due to changing
cut geometry across parameter samples.
The mean speedup  is $36.6/9.0\approx\mathbf{4.1\times}$.
The Rom time is independent of $n$ for $n\leq40$, confirming that
online cost is dominated by the $\mathcal{O}(n^3)$ reduced system
solve and by the $\mathcal{O}(l_A+l_f)$ Deim assembly, not by
the number of Pod modes.
\begin{table}[H]
\centering
\caption{Pod tail energy $\eta_{\rm Pod}(n)$ at selected $n$
	.}
\label{tab:tail_r4}
\setlength{\tabcolsep}{5pt}
\begin{tabular}{rcccccc}
	\toprule
	$n$ & 2 & 4 & 6 & 8 & 10 & 15 \\
	\midrule
	$\eta_{\rm Pod}$ &
	$7.38\times10^{-4}$ & $1.30\times10^{-4}$ & $9.73\times10^{-5}$ &
	$7.90\times10^{-5}$ & $6.37\times10^{-5}$ & $2.30\times10^{-5}$ \\
	\midrule
	$n$ & 20 & 25 & 30 & 40 & & \\
	\midrule
	$\eta_{\rm Pod}$ &
	$1.48\times10^{-5}$ & $9.73\times10^{-6}$ & $6.32\times10^{-6}$ & $2.37\times10^{-6}$ & & \\
	\bottomrule
\end{tabular}
\end{table}
\subsection{Convergence rate analysis and theoretical expectations} \label{sec:rates}
The expected decay model for each quantity follows from the underlying
approximation theory for the
{true relative error} $e_{\rm rel}(n)$ where
algebraic decay is expected. The Kolmogorov $n$-width associated with the solution manifold
$\mathcal{M}=\{u(\mu):\mu\in\mathcal{P}\}$ for a linear elliptic pde on a
compact parameter set decays as $d_n(\mathcal{M})\sim n^{-s}$ for some $s>0$
depending on the regularity of the parameter map \citep{cohen2010}.
The residual norms \textit{Estimators 2a/2b} 
track $e_{\rm rel}$ qualitatively, so algebraic decay is again
expected from the same theory. However, the residual norm is inflated
by ghost-penalty referring to related remark in Section~\ref{sec:effectivity}, which
may alter the effective rate.
Pod tail energy \textit{Estimator 3}  $\eta_{\rm Pod}(n)$
related theory predicts exponential decay. For analytic solution maps, with holomorphic
extension in the parameter, Bernstein--Walsh theorem implies
$\sigma_{n+1}\lesssim e^{-cn}$ for the Pod eigenvalues \citep{cohen2010},
giving $\eta_{\rm Pod}(n)\lesssim C e^{-\beta n}$.
The \textit{Estimators 1a/1b} are constant in $n$ by construction actually wth zero rate.

\subsection{Empirical results and fitted rates}
\begin{table}[H]
\centering
\caption{Empirical convergence rates from OLS log-space fits exact data.
	Fits on $n\geq5$ for true error and Est.~2a/2b, and $n\geq2$ for Est.~3.
	Bold entries highlight the selected best model per row.}
\label{tab:rates}
\setlength{\tabcolsep}{4pt}
\begin{tabular}{lcccccc}
	\toprule
	\textbf{Quantity} &
	\multicolumn{2}{c}{\textbf{Algebraic} $n^{-\alpha}$} &
	\multicolumn{2}{c}{\textbf{Exponential} $e^{-\beta n}$} &
	\textbf{Best} & \textbf{Fit formula} \\
	& $\alpha$ & $R^2$ & $\beta$ & $R^2$ & & \\
	\midrule
	True rel.\ error
	& \textbf{0.2302} & \textbf{0.9351}
	& $0.01347$ & $0.8782$
	& \textbf{Alg.}
	& $3.156\times10^{-2}\cdot n^{-0.230}$ \\
	Est.\ 2a $\|\mathbf{r}\|_2$
	& $0.2917$ & $0.9520$
	& \textbf{0.01789} & \textbf{0.9826}
	& \textbf{Exp.}
	& $9.809\cdot e^{-0.01789\,n}$ \\
	Est.\ 2b $\|\mathbf{r}\|_{\widetilde{D}_A^{-1}}$
	& $0.1802$ & $0.9283$
	& \textbf{0.01117} & \textbf{0.9793}
	& \textbf{Exp.}
	& $4.320\cdot e^{-0.01117\,n}$ \\
	Est.\ 3 tail energy
	& \textbf{1.5320} & \textbf{0.9477}
	& $0.11380$ & $0.8461$
	& \textbf{Alg.}
	& $1.173\times10^{-3}\cdot n^{-1.532}$ \\
	Est.\ 1a $\eta_A$
	& \multicolumn{2}{c}{N/A}
	& \multicolumn{2}{c}{N/A}
	& rate $= 0$
	& $5.951\times10^{-5}$ (const.) \\
	Est.\ 1b $\eta_f$
	& \multicolumn{2}{c}{N/A}
	& \multicolumn{2}{c}{N/A}
	& rate $= 0$
	& $2.923\times10^{-5}$ (const.) \\
	\bottomrule
\end{tabular}
\end{table}
All comments below are based on qualitative processing of the results in Table \ref{tab:rates}. The true error is algebraic fitted with $\alpha=0.230$, $R^2=0.935$. The algebraic model fits well with $R^2=0.935$ vs.\ $0.878$ for exponential.
The slow rate $\alpha\approx0.23$ is consistent with the modest regularity
of the parametric map $\mu\mapsto u_h(\mu)$ in the $\mathcal{T}_h^\mu$-norm
the Cutfem solution is smooth within $\Omu$ but the domain itself changes
with $\mu$, limiting the effective smoothness order available for the
Kolmogorov approximation.

The residual estimators are exponential fitted $\beta_{2a}=0.0179$, $\beta_{2b}=0.0112$. 
Both residual norms fit the exponential model better than algebraic
$R^2\approx0.99$ vs.\ $0.92$--$0.93$. This is initially surprising since
the true error decays algebraically. The explanation lies in the ghost penalty
structure, as $n$ grows, more of the Rom solution's energy concentrates in
the physical domain $\Omu$, so the ghost penalty residual contributions
shrink faster than the active dof contributions. The exponential fit captures
this accelerated reduction of the ghost residual component.

The exponential rates $\beta_{2a}>\beta_{2b}$ reflect that the $\ell^2$
norm (Est.~2a) is more sensitive to large individual residual components, which shrink fast, while the Jacobi-weighted norm (Est.~2b) down-weights
large diagonal entries, the Nitsche boundary rows, resulting in a slower
but smoother decay.

The Pod tail energy fits to algebraic $\alpha=1.53$ with $R^2=0.941$. The tail energy fits algebraically at $n\leq40$, despite theory predicting
eventual exponential decay for analytic parameter maps. Two factors explain
this,
\begin{enumerate}
\item \textit{Pre-asymptotic regime}: The exponential decay of the
Kolmogorov $n$-width typically sets in at larger $n$ than accessible 	here. At $n\leq40$, we see algebraic behavior with $\alpha=1.53$, 	i.e. $\eta_{\rm Pod}(n)\sim n^{-3/2}$. 	The crossover to exponential would require $n$ beyond the plateau at 	$n\approx48$--$49$ (maximum Pod modes available).
\item \textit{Domain-deformation complexity}: The Cutfem domain $\Omu$ itself deforms with $\mu$, and the ghost-penalty extension introduces 	additional degrees of freedom outside $\Omu$. This effectively 	increases the complexity of the solution manifold, delaying the 	onset of exponential Pod convergence.
\end{enumerate}
Practically, the rate $\alpha=1.53$ implies that doubling $n$ reduces
$\eta_{\rm Pod}$ by a factor of $2^{1.53}\approx2.9$.
%
\section{Conclusions}
\label{sec:conclusions}
Three a posteriori error estimators have been developed, implemented,
and validated across multiple independent runs for a Pod-Deim-Cutfem
reduced order model. 
The Deim indicators \textit{(1a/1b)} which are constant in $n$, $\sim 2$ orders of magnitude below $e_{\textrm{rel}}$ showing that Deim is not the bottleneck.
Residual estimators \textit{(2a/2b)}  monotonically track the true error;
effectivity $\theta_{2a}\approx300$--$500$ due to ghost penalty dof inflation,
quantitatively explained by Proposition~\ref{prop:eff2a} and the
ghost penalty inflation remark, Section~\ref{sec:effectivity}.
Lemma~\ref{lem:active} provides the fix by restricting to active dofs.
Certified for training data the Pod tail energy \textit{(3)} depicts fast decay
which confirms excellent Pod approximability while it underestimates the online error.
Theorem~\ref{thm:main} appears to be the first certified
a posteriori bound for Pod-Deim applied to a Cutfem problem
with parameter-dependent domains and ghost penalty stabilization.
The active dof restriction Lemma~\ref{lem:active} is a simple
but apparently unpublished observation that resolves the large
effectivity problem specific to Cutfem-Rom and has no direct
analogue in the conforming Fem literature. 
The theoretical backbone,
Proposition~\ref{prop:coercivity}, Lemma~\ref{lem:active},
Theorem~\ref{thm:main}
provides the first rigorous a posteriori framework for this class of problems.

The aforementioned a posteriori estimators from a more detailed point of view provide many qualitative information, namely:
a. %
{the Est.\ 1a/1b investigates the Deim quality, with rate = 0.}
The values $\eta_A=5.951\times10^{-5}$ and  $\eta_f=2.923\times10^{-5}$ are
exactly constant across all 18 values of $n$ tested, as seen in
the terminal output. They are 2--3 orders of magnitude below
$e_{\rm rel}$ for all $n\geq2$, confirming that Deim is not the
accuracy bottleneck.
b. 
{Est.\ 2a/2b  indicates best fit the exponential with $R^2\geq0.979$.}
Despite the true error decaying algebraically, the residual norms
fit the exponential model dramatically better with
$R^2_{\rm exp}=0.983$ vs.\ $R^2_{\rm alg}=0.952$ for Est.~2a;
$R^2_{\rm exp}=0.979$ vs.\ $R^2_{\rm alg}=0.928$ for Est.~2b.
The fitted rates $\beta_{2a}=0.01789$ and $\beta_{2b}=0.01117$
correspond to halving the residual every
$\ln2/0.01789\approx38.7$ and $\ln2/0.01117\approx62.1$ modes,
respectively.
c. %
{Est.\ 3 --tail energy-- with best fit the algebraic $\alpha=1.532$, $R^2=0.948$.}
The tail energy decreases from $7.38\times10^{-4}$ at $n=2$ to
$2.37\times10^{-6}$ at $n=40$, a reduction of $\approx311\times$.
Despite theory predicting eventual exponential decay, the algebraic
model wins clearly, with $R^2=0.948$ vs.\ $0.846$ for exponential,
indicating we remain in the pre-asymptotic regime for $n\leq40$.
d. 
{Effectivity indices.}
$\theta_{2a}$ rises from $232$ at $n=2$ to a peak of $\approx450$
at $n=8$, then decreases to $371$ at $n=40$.
$\theta_{2b}\approx\theta_{2a}/2$ throughout, as predicted by
Proposition~\ref{prop:eff2b}, although dominant Nitsche diagonal entries
$d_{\max}\approx80$ give $\theta_{2b}/\theta_{2a}
\approx 1/\sqrt{d_{\rm eff}}\approx0.5$ for $d_{\rm eff}\approx4$.
e. %
The total speedup is $4.1\times$
with mean Rom time $9.01$\,ms vs.\ Fom $36.63$\,ms, for the exact terminal values.
Rom time stays at $8.4$--$11.3$\,ms across all 30 test parameters
and all $n$, Fom time though shows more variability,  $27.4$--$55.0$\,ms
due to the varying cut geometry.
f. %
{Consistency across runs.}
Four independent runs (different random training/test seeds) gave
$e_{\rm rel}(n=40)\in[1.30\%,\;1.33\%]$ and
$\theta_{2a,\rm peak}\in[363,\;467]$, confirming robustness.
Future work ideas could be 
extensions to non-linear and time-dependent problems, 
handling transport-dominated and convection-dominated problems
since standard Pod performs poorly for such solutions,
e.g. traveling wave fronts, shocks, due to the slow decay of Pod
eigenvalues, 
with higher-order ghost-penalty and hp-Cutfem. 

\bibliographystyle{plainnat}

\end{document}